\documentclass[11pt]{article}

\usepackage{grffile}

\usepackage{amsmath,amsfonts,amscd,a4wide}
\usepackage{amsmath}
\usepackage{amssymb}
\usepackage{graphicx}
\usepackage{amsthm}
\usepackage{color}
\usepackage{color}

\usepackage{color}
\usepackage{transparent}
\usepackage[skip=0pt,font=footnotesize]{caption}
\usepackage{subcaption}
\usepackage{enumitem}

\usepackage[margin=1in]{geometry}
\usepackage{geometry}
\newtheorem{Lemma}{Lemma}[section]
\newtheorem{Theorem}{Theorem}
\newtheorem{Proposition}[Lemma]{Proposition}

\newtheorem{Remark}[Lemma]{Remark}
\newtheorem{Definition}[Lemma]{Definition}
\newtheorem{Hypothesis}{Hypothesis}

\newenvironment{Proof}%
 {\begin{trivlist} \item[]{\bf Proof. }}%
 {\hspace*{\fill}$\rule{.4\baselineskip}{.4\baselineskip}$\end{trivlist}}

 {\begin{trivlist}\item[]\textbf{Acknowledgments.}}{\end{trivlist}}

\makeatletter\@addtoreset{figure}{section}\makeatother

\makeatletter \@addtoreset{equation}{section} \makeatother

\newcommand{\R}{\mathbb{R}}
\newcommand{\C}{\mathbb{C}}

\newcommand{\Z}{\mathbb{Z}}

        \newcommand{\mc}[1]{\mathcal{#1}}
        \newcommand{\mb}[1]{\mathbb{#1}}
         \newcommand{\mf}[1]{\mathfrak{#1}}
        \newcommand{\tl}[1]{\tilde{#1}}
        \newcommand{\lp}{\left}
        \newcommand{\rp}{\right}
                \newcommand{\la}{\lp\langle}
        \newcommand{\ra}{\rp\rangle}
        \newcommand{\beq}{\begin{equation}}
        \newcommand{\eeq}{\end{equation}}
        \newcommand{\ba}{\begin{align}}
        \newcommand{\ea}{\end{align}}
        \newcommand{\fr}[2]{\frac{#1}{#2}}
        \newcommand{\p}{\partial}
        \newcommand{\ri}{\mathrm{i}}
        
        \newcommand{\rlin}{\mathrm{lin}}
        
        \newcommand{\re}{\mathrm{e}}

        \newcommand{\rre}{\mathrm{Re}}
        \newcommand{\rim}{\mathrm{Im}}
        
         \newcommand{\rl}{\mathrm{l}}
        
                \newcommand{\rtf}{\mathrm{tf}}

                        \newcommand{\rs}{\mathrm{s}}
                        \newcommand{\rsu}{\mathrm{su}}
                                \newcommand{\rss}{\mathrm{ss}}
                                       \newcommand{\rcu}{\mathrm{cu}}
                                                \newcommand{\ru}{\mathrm{u}}
                                                \newcommand{\rff}{\mathrm{ff}}
                                                \newcommand{\rpr}{\mathrm{pr}}
        \newcommand{\gso}{\gamma_\rss}
        \newcommand{\guo}{\gamma_\rsu}
                \newcommand{\gst}{\kappa_\rss}
        \newcommand{\gut}{\kappa_\rsu}

        \makeatletter
\def\blfootnote{\gdef\@thefnmark{}\@footnotetext}
\makeatother

\begin{document}

\vspace*{0.4in}
\begin{center}
{\fontsize{16}{16}\textbf{Pattern formation in the wake of triggered pushed fronts}}\\[0.2in]

Ryan Goh and Arnd Scheel\blfootnote{*Research partially supported by the National Science Foundation through grants NSF- DMS-0806614 and NSF-DMS-1311740. This material is based upon work supported by the National Science Foundation Graduate Research Fellowship under grant NSF-GFRP-00006595}\\
\textit{\footnotesize University of Minnesota, School of Mathematics,   206 Church St. S.E., Minneapolis, MN 55455, USA}\\

September 22th, 2015
\end{center}

\abstract{  
\noindent
Pattern-forming fronts are often controlled by an external stimulus which progresses through a stable medium at a fixed speed, rendering it unstable in its wake. By controlling the speed of excitation, such stimuli, or ``triggers," can mediate pattern forming fronts which freely invade an unstable equilibrium and control which pattern is selected. In this work, we analytically and numerically study when the trigger perturbs an oscillatory pushed free front.  In such a situation, the resulting patterned front, which we call a pushed trigger front, exhibits a variety of interesting phenomenon, including snaking, non-monotonic wavenumber selection, and hysteresis. Assuming the existence of a generic oscillatory pushed free front, we use heteroclinic bifurcation techniques to prove the existence of trigger fronts in an abstract setting motivated by the spatial dynamics approach.  We then derive a leading order expansion for the selected wavenumber in terms of the trigger speed. Furthermore, we show that such a bifurcation curve is governed by the difference of certain strong-stable and weakly-stable spatial eigenvalues associated with the decay of the free pushed front. We also study prototypical examples of these phenomena in the cubic-quintic complex Ginzburg Landau equation and a modified Cahn-Hilliard equation.
}

{\small
{\bf Running head:} {Triggered pushed fronts}

}
\vspace*{0.2in}

\section{Introduction}

Over the past three decades, much experimental, numerical, and theoretical work has been done to study pattern formation in the wake of invading fronts (see \cite{ben1985pattern,cross1993pattern,Dee83,vanSaarloos03} for some of the first papers and reviews).  Typically, such fronts arise via the nucleation and invasion of an instability into a homogeneous unstable equilibrium which leaves behind a pattern forming front in its wake.  In many cases, the resulting pattern is an unstable periodic wave train.  The wavenumber of this periodic pattern is usually independent of the initial perturbation and is selected by the nonlinear front propagation. We refer to such fronts, which mediate this invasion process, as \emph{free fronts}.  

	In practice, such invasion processes can be difficult to control as uniform suppression of random fluctuations is required to prepare the unstable equilibrium and hence form defect-free patterns. Also, one may only have limited control on system parameters, making it even more difficult to control the patterns formed. One way of gaining control of the pattern-forming process is to use an external mechanism which travels through a stable medium and locally excites it into an unstable state.  We shall call such a mechanism a \emph{trigger} and the resulting front, which connects the unstable and stable states, a \emph{preparation front}. Once the unstable state is established, the mechanism which governs the free front, causes a uniformly patterned state to nucleate in the wake of the trigger.  We shall call the resulting pattern-forming front a \emph{trigger front}.  See Figure \ref{f:sch} below for a schematic description of this process.
	
	Heuristically, one can think of the trigger as an effective boundary condition for the system when posed in a co-moving frame. Stationary solutions in this coordinate frame are usually referred to as nonlinear global modes \cite{chomaz2005global}. They mediate the transition from convective to absolute instability in a semi-infinite domain. From this perspective, our problem is somewhat equivalent to problems studied in \cite{couairon1997cf,chomaz1999against,pier1998steep}, and our results can be understood as a rephrasing and improvement of expansions in \cite{couairon2001pushed}. In particular, we emphasize universality in expansions for wavenumber and frequency in terms of only properties of the corresponding free front. We also note that a slightly different but related approach was used in \cite{doelmangeometric15} to study the effect of defects on one-dimensional localized structures.
	
Many examples of this triggered pattern formation arise in systems with mass-conserving properties. Typical model equations for such systems are the Cahn-Hilliard equation \cite{Foard12,Kopf14, krekhov,lagzi13}, the Keller-Segel model for chemotaxis \cite{mimura, matsuyama}, reaction-diffusion systems \cite{manz2002excitation}, or phase-field systems \cite{fife2002pattern,glasner2000traveling,Goh11}.  Other examples arise in ion-bombardment studies \cite{bradleygelfand} and are modeled by Kuramoto-Shivashinsky-type models, while still others arise when studying general osciliatory instabilities using real and complex Ginzburg-Landau models \cite{chomaz2005global,couairon1997,couairon1997cf}.

\paragraph{Free Fronts: Pushed vs. Pulled} 
Since the pattern formed in the wake of a trigger front is controlled by how the free front interacts with the trigger, we must discuss free fronts in more detail before we can describe our results. In an analogous fashion to super- and sub-critical transitions, free fronts come in two generic types known as \emph{pulled} and \emph{pushed}.

Pulled fronts can be described to very good approximation by a linear analysis based on branch points of a complex dispersion relation. The speed of such a front, known as the linear spreading speed, is determined by a marginal stability criterion \cite{Dee83}, which requires that the trivial state is pointwise marginally stable in a frame moving with this speed. In other words, decreasing the speed of the frame of observation, the instability of the trivial state changes from convective to absolute. Furthermore, invasion of such fronts is governed by linear growth in the leading edge which then saturates in the wake due to the nonlinearities of the system. Such nonlinear pulled fronts are known to be very sensitive to disturbances in the leading edge. Convergence towards, as well as relaxation of small perturbations to, pulled fronts is typically slow, in fact algebraic.  In a co-moving frame, this type of invasion can be either stationary or oscillatory with frequency $\omega_\mathrm{fr}$.  This frequency is typically a fraction of the frequency derived by the marginal stability criterion, $\omega_\mathrm{fr}=\omega_\mathrm{lin}/\ell$. The case of strong resonance, $\ell=1$, is often referred to as node conservation in the leading edge.  For a systematic study of such fronts using pointwise Green's functions see \cite{holzer14}.

Pushed fronts arise when nonlinearities amplify linear growth sufficiently so that the speed of propagation of disturbances exceeds the linear, pulled speed. Pushed fronts are generally steeper and convergence towards them is fast, being exponential in time. In fact, the Green's function for the linearization at pushed fronts exhibits simple poles associated with the neutral Goldstone modes, while the linearization near pulled fronts exhibits a singularity with structure similar to the 3-dimensional heat kernel \cite{gallay94}.  In the language of spatial dynamical systems, such a front consists of a heteroclinic orbit which converges to an equilibrium along a strong-stable invariant manifold.  While numerical studies of such fronts have been performed in many different systems, rigorous theoretical study has been limited to a small number of mathematical models including the Nagumo equation, coupled-KPP equations,  Lotka-Volterra systems, and the Complex Ginzburg Landau equation \cite{holzer2012slow, holzer14b, van1989,vanSaarloos92}. 

For a more comprehensive review of these two generic types of fronts, including many numerical and physical examples see \cite{vanSaarloos03}.

\paragraph{Our Contributions}
The main question of interest here is the effect of the trigger on the wavenumber of the periodic pattern in the wake of the front. One observes that for large enough trigger speeds the influence is negligible, possibly after an initial transient: in the wake of the trigger, one observes a front very close to the free front, and the distance between trigger and patterns increases linearly in time.   We will thus focus on the situation when the speed of the trigger is close to that of the free front.  One then expects small corrections to the wavenumber in the wake. In particular, for trigger speeds smaller than the speed of a free front, one expects a locked state, where the free front has caught up with the trigger.  In this locked state, the trigger can be thought of as exerting a pressure, or strain on the periodic pattern, causing a perturbation in its wavenumber.

\begin{figure}[h!]
\centering
\includegraphics[width=0.85\textwidth]{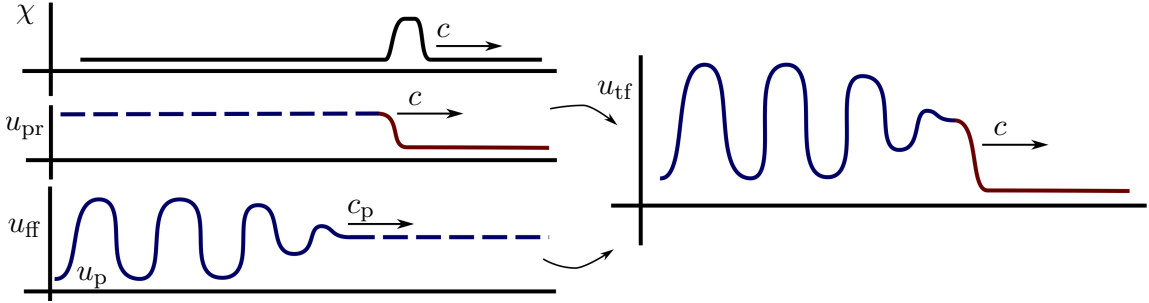}
\vspace{-0.0in}
\caption{Schematic depiction of our results.  The triggering mechanism, $\chi$, travels with speed $c$ and creates the preparation front $u_\mathrm{pr}$ connecting a stable state (solid red) with an unstable state (dashed blue). The free front, $u_\mathrm{ff}$, invades the unstable state with speed $c_\mathrm{p}$, leaving behind a periodic pattern $u_\mathrm{p}$.  The resulting pattern-forming trigger front, $u_\mathrm{tf}$, is obtained from the combination of $u_\mathrm{pr}$ and $u_\mathrm{tf}$,  for speeds $c$ close to $c_\mathrm{p}$.}\label{f:sch}
\end{figure}

In the previous work \cite{GohScheel14}, we rigorously characterized a prototypical example of a trigger front which is perturbed from a \emph{pulled} free front.  There it was shown that wavenumber and front position asymptotics can be predicted to leading order by the absolute spectrum of the unstable trivial state and to second order by the projective distance between two invariant manifolds near the unstable homogeneous equilibrium. Furthermore, it was shown in this case that the wavenumber of the periodic pattern in the wake varies monotonically as the speed of the trigger is varied.

In this work, our goal is to study trigger fronts perturbed from a \emph{pushed} free front.  Conceptually, our results are as follows.  Assume that a one-dimensional, evolutionary pattern forming system has the following properties:

\begin{itemize}
\item There exists an oscillatory pushed free front $u_\mathrm{ff}$ invading an unstable homogeneous equilibrium $u_*$ with speed $c_\mathrm{p}>0$.
\item For speed $c_\mathrm{p}$, there exists a preparation front $u_\mathrm{pr}(x - c_\mathrm{p}t)$ formed in the wake of a spatial trigger which connects $u_*$ to a stable homogeneous state $\tl u_*$ as $\xi:=x - c_\mathrm{p}t$ increases from $-\infty$ to $+\infty$.  
\item  The fronts $u_\rff$ and $u_\rpr$ are generic.  In other words, when viewed as heteroclinic orbits in a spatial dynamics formulation, $u_\rpr$ is transverse while  $u_\rff$ is transversely unfolded in parameters $\omega$ and $c$, where $\omega$ is the temporal frequency of the periodic pattern associated with $u_\rff$ and $c$ is the speed  of the trigger.
\item The inclination properties of the relevant invariant manifolds about $u_\rpr$ are generic.
\end{itemize}

Then for trigger speeds close to the free invasion speed $c_\mathrm{p}$, there exists a family of pushed trigger fronts connecting a spatially periodic orbit to the aforementioned stable state. Moreover, this family has a bifurcation curve in the parameter space $\mu := (c - c_\mathrm{p}, \omega - \omega_\mathrm{p})\in\R^2$, with the asymptotic form
\beq
\mu(L) = K \re^{\Delta \nu L}(1+\mc{O}\lp(\re^{-\delta L})\rp),
\eeq
where, $L\gg1$, $K$ is a linear mapping from $\C$ to $\R^2$, and $\Delta \nu$ denotes the difference of strong stable eigenvalues associated with the decay of the free pushed front, and other weakly stable eigenvalues (see Figure \ref{f:spec}).

\begin{figure}[h!]\label{f:spiral}
\centering
\includegraphics[width=0.75\textwidth]{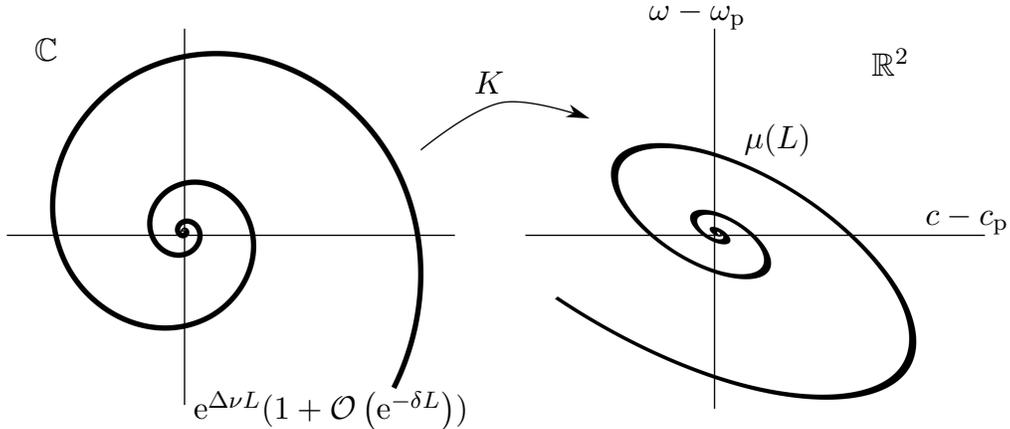}
\caption{Leading order bifurcation curve of pushed trigger fronts in $\mu$-parameter space.}
\end{figure}

If a pushed free front is oscillatory, the bifurcation curve of pushed trigger fronts takes on a logarithmic spiral shape, leading to a variety of interesting phenomena. Namely, such trigger fronts will exhibit snaking behavior. This leads to the possibility of multi-stability of fronts, locking behavior for trigger speeds slightly higher than $c_\mathrm{p}$,  and finally hysteretic switching between different wave-numbers. This last effect is particularly interesting as it could potentially be exploited in the design and control of self-organized patterning processes.  

The genericity assumptions above can also be formulated in terms of spectral information.  For $u_\rpr$, such hypotheses are equivalent to assuming that the Evans function associated with the linearization about the front has no zeros at the origin.  For $u_\mathrm{ff}$, the associated Evans function has a zero of algebraic multiplicity two at the origin, corresponding to the temporal and spatial translation symmetries of the front.

Technically, we use an abstract formulation motivated by the spatial dynamics approach and employ heteroclinic matching techniques to prove existence of pushed trigger fronts and give universal asymptotics for their frequency and wave-numbers.  We shall show that the front dynamics are, to leading order, governed by the spectral gap between strong-stable spatial eigenvalues, which govern the asymptotic decay of the free pushed front, and other weakly stable eigenvalues.  As seen in Figure \ref{f:spec}, the simplest form of such a gap may come in several varieties, each of which may lead to different phenomena.  In this paper we shall focus on the case depicted on the left where the gap is determined by two complex conjugate pairs.  The other cases in this figure may also lead to many interesting phenomenon and are briefly discussed in Section \ref{s:dis}.  

Throughout the paper, we consider two prototypical examples to elucidate our results.  The first of these is the cubic-quintic complex Ginzburg-Landau equation (qcGL).  We choose this relatively simple example to demonstrate our results and motivate their application to more complicated systems.  Finding pushed trigger fronts in the qcGL equation can be reduced to a finite dimensional traveling-wave ODE in which all of the required hypotheses for our result have been proven in previous studies, or can be obtained by straightforward arguments.  

The second example we consider is a modified Cahn-Hilliard equation. This equation will serve as an illustration for how our results apply in the case where the existence problem is inherently infinite-dimensional. While in this setting, it is not straightforward to verify the required hypotheses (see Section \ref{ss:ch} and Section \ref{ss:ar}), we provide numerical evidence showing the predicted phenomenon, and also evidence for one of our most important hypotheses: the existence of an oscilliatory pushed free front.

\begin{figure}\label{f:spec}
\centering
\includegraphics[width=0.85\textwidth]{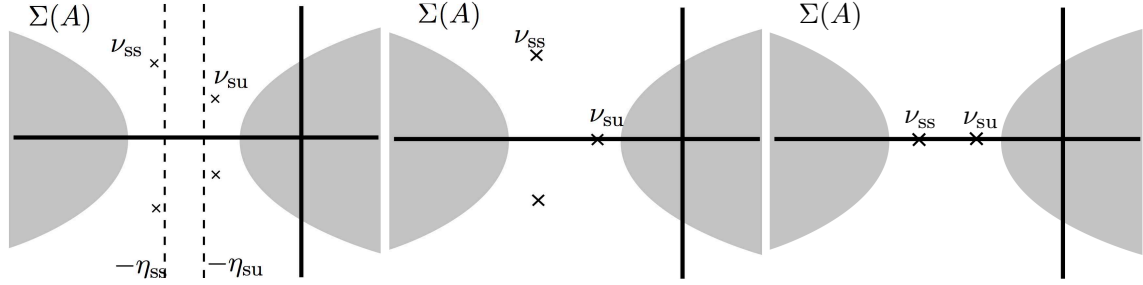}
\caption{Depiction of different cases for splitting of spatial eigenvalues corresponding to a free pushed front. The grey areas denote the rest of the spectrum of the linearization of the spatial dynamics formulation. We study the case depicted in the left plot, here the dotted lines denote the exponential weights we use to select relevant solutions near the origin. }\label{f:spec}
\end{figure}

\paragraph{Outline}
The following work is structured as follows.  In Section \ref{s:ex} we give examples of the relevant phenomena in specific equations.  In Section \ref{s:ab} we formulate our abstract hypotheses and state our main result. In Section \ref{s:pf} we then give the heteroclinic matching proof for the existence of pushed trigger fronts and obtain leading order expansions for the bifurcation curve in terms of the spectral information of the system.  We conclude our work in Section \ref{s:dis} by discussing future areas of work and how our results could be improved and extended.

\section{Examples and sketch of main result}\label{s:ex}
To motivate our results, we briefly describe examples in the cubic-quintic complex Ginzburg-Landau and Cahn-Hilliard equations which illustrate the phenomena mentioned in the introduction.  

\subsection{Complex Ginzburg-Landau equation}\label{ss:cgl}  

The complex Ginzburg-Landau equation has been used as a modulation equation to study the onset of coherent structures in many physical systems.   Furthermore, pattern-forming free fronts have been extensively studied in this setting.  In particular, it has been shown in \cite{vanSaarloos92} that the cubic-quintic variant
\beq\label{e:cgl}
\tl u_t = (1+\ri \alpha) \tl u_{xx} + \tl u + (\rho + \ri \gamma) \tl u|\tl u|^2 - (1+\ri \beta) \tl u |\tl u|^4, \qquad x,t \in \R, \,\,\, \tl u\in \C,
\eeq
possesses pushed free invasion front solutions for a range of parameters $\rho,\alpha, \gamma,\beta$.  That is, there exist front solutions which connect a wave train $u_p(x,t) = r\re^{\ri(k_\mathrm{p}x - \omega_\mathrm{p} t)}$ at $x= -\infty$ to the unstable homogeneous equilibrium $u_*\equiv 0$ at $x\rightarrow \infty$ with an interface which invades the unstable state $u_*$ with a speed, $c_\mathrm{p}$, faster than the linearized dynamics predict.  The parameters of the asymptotic wave train, $r,k,\omega\in \R$, can be found to satisfy the nonlinear dispersion relation 
\begin{align}
1 &= k^2 - \rho r^2 + r^2,\notag\\
\omega + c k &= \alpha k^2 - \gamma r^2 + \beta r^4.\label{e:disp}
\end{align}

By shifting into a co-moving frame $\xi = x - c_\mathrm{p}t$ and detuning by $ u = \re^{\ri \omega_\mathrm{p} t} \tl u$, such a traveling front takes the form of a heteroclinic orbit in the finite-dimensional system
\beq\label{e:cgl2}
0 =  (1+\ri \alpha)  u_{\xi\xi} + c_\mathrm{p}u_\xi+  (1 - \ri \omega) u + (\rho + \ri \gamma)  u| u|^2 - (1+\ri \beta)  u | u|^4.
\eeq

In this setting, we can then study how a spatially progressive triggering mechanism, $\chi_\epsilon$, affects this pattern-forming front using the following system
\begin{align}
0 &= (1+\ri \alpha) u_{\xi\xi} +cu_\xi+ (\chi_\epsilon(\xi) - \ri \omega )u + (\rho + \ri \gamma)  u| u|^2 - (1+\ri \beta) u |u|^4,\label{e:cgl3}\\
\epsilon\fr{d}{d\xi}\chi_\epsilon &= \chi_\epsilon^2-1,\label{e:cgl3a}
\end{align}
where $\chi_\epsilon$ takes the role of the trigger with $0<\epsilon<<1$ and $\chi_\epsilon(0) =0$.   When viewed in the stationary coordinate frame, the inhomogeneity $\chi_\epsilon$ travels through the spatial domain, altering the PDE-stability of $u_*$. For $\xi>0$ the state is stable, while for $\xi<0$ it is unstable. Thus, if $u_0$ is locally perturbed, an oscillatory instability will develop, leading to the formation of a patterned state in the wake of $\chi_\epsilon$.  

Numerical simulations show that such a mechanism creates pattern-forming fronts which behave in a strikingly different manner than in the pulled case \cite{GohScheel14}.  As can be seen in Figure \ref{f:cgl}, the front exhibits snaking behavior as the trigger speed $c$ is varied near $c_\mathrm{p}$.   Here the front interface of the solution hysteretically ``locks" at different distances to the trigger interface located at $x - ct = 0$. Furthermore, this locking causes trigger fronts to persist for speeds larger than the free invasion speed $c_\mathrm{p}$.  Additionally, Figure \ref{f:cgl} shows wavenumbers of the periodic pattern in the wake of the trigger vary non-monotonically and hysteretically as the trigger speed $c$ is varied. 

 \begin{figure}[h!]
\centering
\begin{subfigure}[h]{0.5\textwidth}
\includegraphics[width=\textwidth]{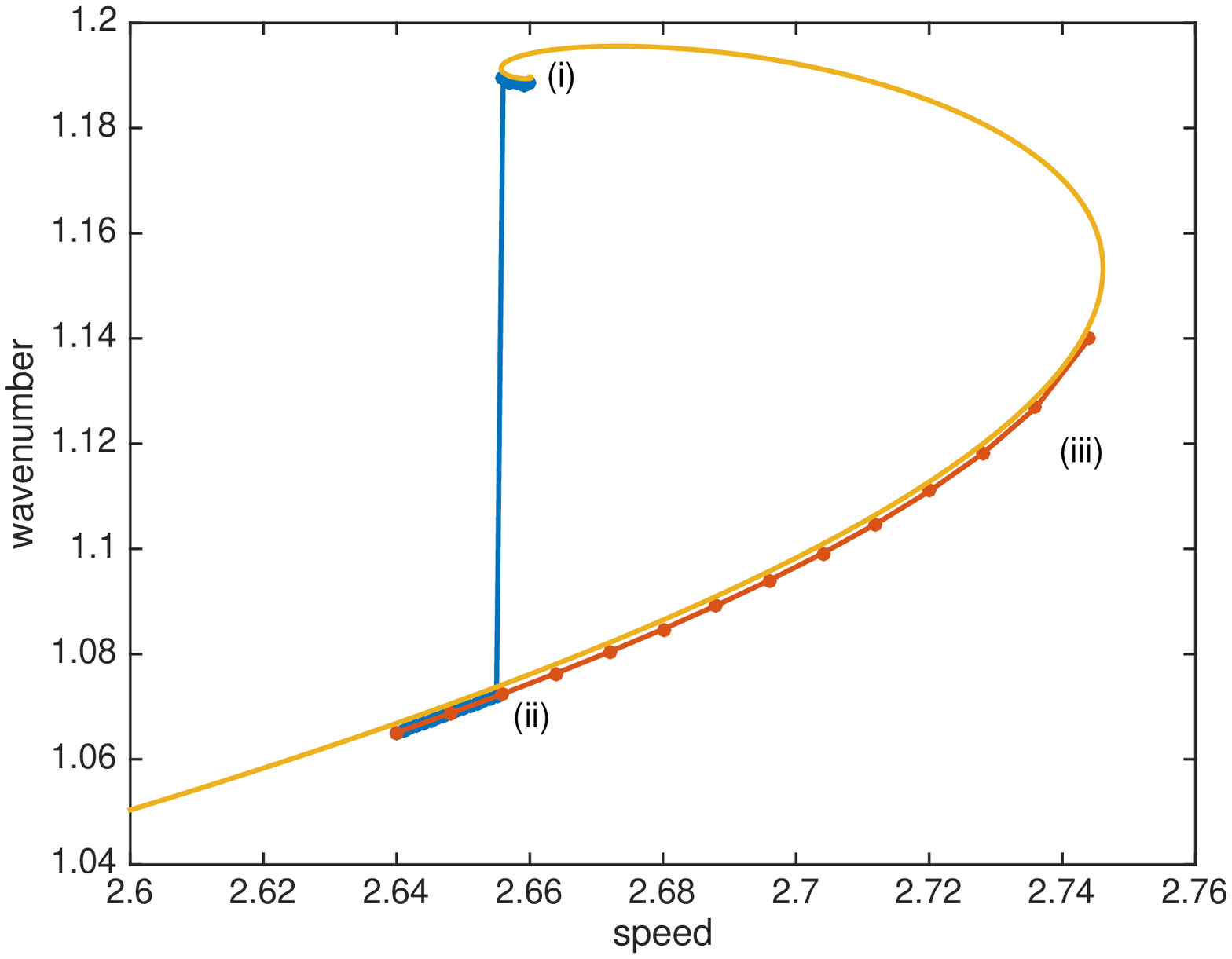}
\end{subfigure}
\hspace{-0.2in}
\begin{subfigure}[h]{0.5\textwidth}
\includegraphics[width=\textwidth]{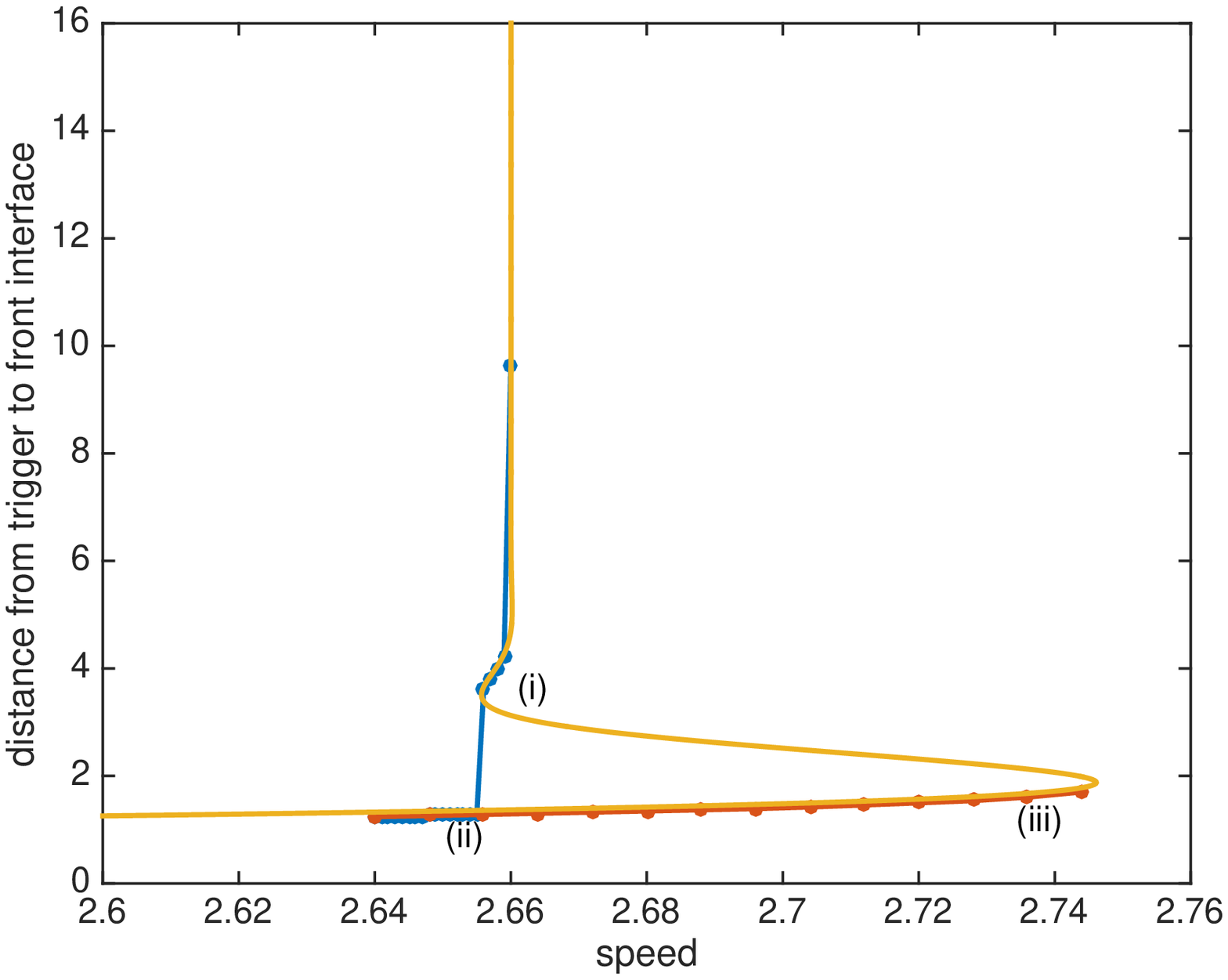}
\end{subfigure}

\begin{subfigure}[h]{1.1\textwidth}
\hspace{-0.5in}
\includegraphics[width=\textwidth]{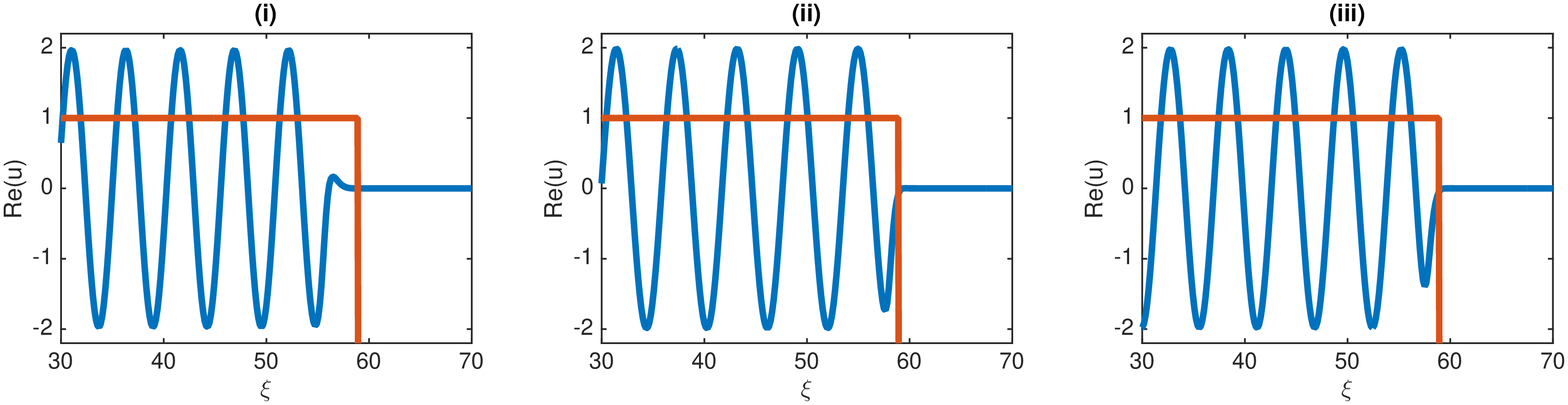}
\end{subfigure}

\caption{Numerical bifurcation diagrams comparing computations of triggered qcGL equation from AUTO07P (yellow) and direct simulation (blue and orange dots) with parameter values $\alpha = 0.3, \gamma = -0.2, \beta = 0.2, \rho = 4$ so that $c_p \approx 2.66$. The bottom three figures depict triggered pushed front profiles for a range of parameter values: (i): $(c,k) = (2.656,1.1894)$, (ii): $(c,k) = (2.646,1.0678)$,  (iii) $(c,k) = (2.728,1.1181)$, zoomed in near the trigger to illustrate the distance to the interface of the trigger $\chi_\epsilon$ which is overlaid in orange.  The direct simulations were done using a 2nd-order exponential time differencing scheme  (see \cite{Cox02}) with $dt = 0.01$, a spectral spatial discretization with $2^{10}$ Fourier modes, and were run in the co-moving frame with speed $c$.  Note that the trigger was made negative near the left boundary at $\xi = 0$ (not pictured) to accommodate for the periodic boundary conditions. This was not found to affect the results as the nucleated patterns were unaffected by this interface, having negative group velocity.}

\label{f:cgl}
\end{figure}
 
These results were obtained using both numerical continuation and direct simulation of \eqref{e:cgl3}. The yellow curves were found via numerical continuation in AUTO07P. In order to avoid periodic boundary conditions, these computations were done in the blow-up coordinates derived in \cite{Goh11}, where periodic orbits in the traveling wave equation collapse to equilibrium points.  The dotted lines (blue and orange) come from measurements of direct simulations.  In these simulations the homogeneous state $u_*$ was locally perturbed far away from the trigger interface, resulting in a patterned state which locked some distance away from the interface (blue curve).  The trigger speed $c$ was then adiabatically decreased and, when $c$ reached the turning point of the bifurcation curve found using AUTO07P, the front detached and re-locked to a solution branch with a different wavenumber and front interface closer to the trigger.  The trigger speed $c$ was then adiabatically increased, continuing solutions along this different branch (orange curve).

\subsection{Cahn-Hilliard equation}\label{ss:ch}

We have also have investigated these types of fronts in a modified Cahn-Hilliard equation
\beq\label{e:CH}
u_t = -(u_{xx} + f(u))_{xx},\quad f(u):= u + \gamma u^3 - u^5,\quad x,t,u\in \R.
\eeq
Because the linearization about the homogeneous unstable state $u_*\equiv 0$ is the same as the standard Cahn-Hilliard equation with $f(u) = u - u^3$, \eqref{e:CH} will have the same linear spreading speed \cite{vanSaarloos03},
$$
c_\rlin = \fr{2}{3\sqrt{6}} \lp(2+\sqrt{7}\rp)\sqrt{\sqrt{7} - 1}.
$$
Direct numerical simulations using both spectral and finite-difference methods have suggested that, for $\gamma>0$ sufficiently large this equation possesses oscillatory pushed invasion fronts which freely invade the homogenous state $u_*$.  Figure \ref{f:chpp} depicts spacetime diagrams of two free invasion fronts in \eqref{e:CH}, one with $\gamma<0$ and one with $\gamma>0$.  In the former case the front approximately travels with the linear speed $c_\rlin$, while in the latter case the front travels with a faster speed and possesses steeper decay at the leading edge.  Since the Cahn-Hilliard equation cannot be detuned as in the CGL equation above, such pushed front solutions would arise in a co-moving frame of speed $c_\mathrm{p}$ as time-periodic solutions with some temporal frequency $\omega_\mathrm{p}$.  That is they are solutions to the equation
\beq
\omega_\mathrm{p} u_\tau = -(u_{\xi\xi} + f(u))_{\xi\xi} +c_pu_\xi, \quad \xi\in \R, \tau\in [0,2\pi].
\eeq

We can then study pushed trigger fronts by introducing a uniformly-translating spatial trigger as above 
\beq\label{e:mCH}
\omega u_\tau = -(u_{\xi\xi} +\tl f(\xi,u))_{\xi\xi}+cu_\xi,\quad \tl f(\xi,u):=\chi(\xi) u+\gamma u^3 - u^5,\quad \epsilon \fr{d}{d\xi}\chi_\epsilon = \chi_\epsilon^2 -1 
\eeq
with $\epsilon>0$ small, and $\chi_\epsilon(3N/4) = 0$ for some $N>0$.  

\begin{figure}
\includegraphics[width=0.45\textwidth]{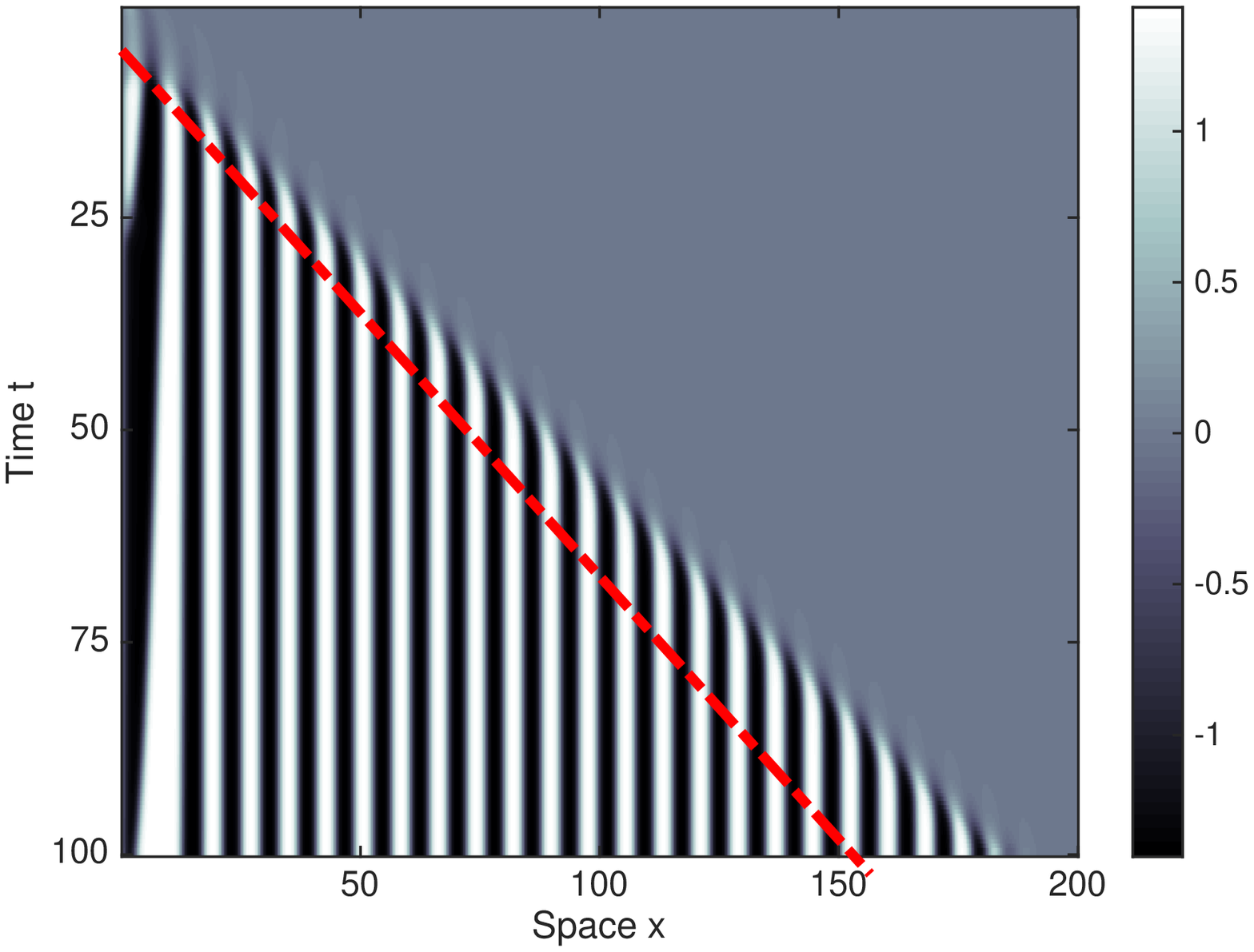}
\includegraphics[width = 0.45\textwidth]{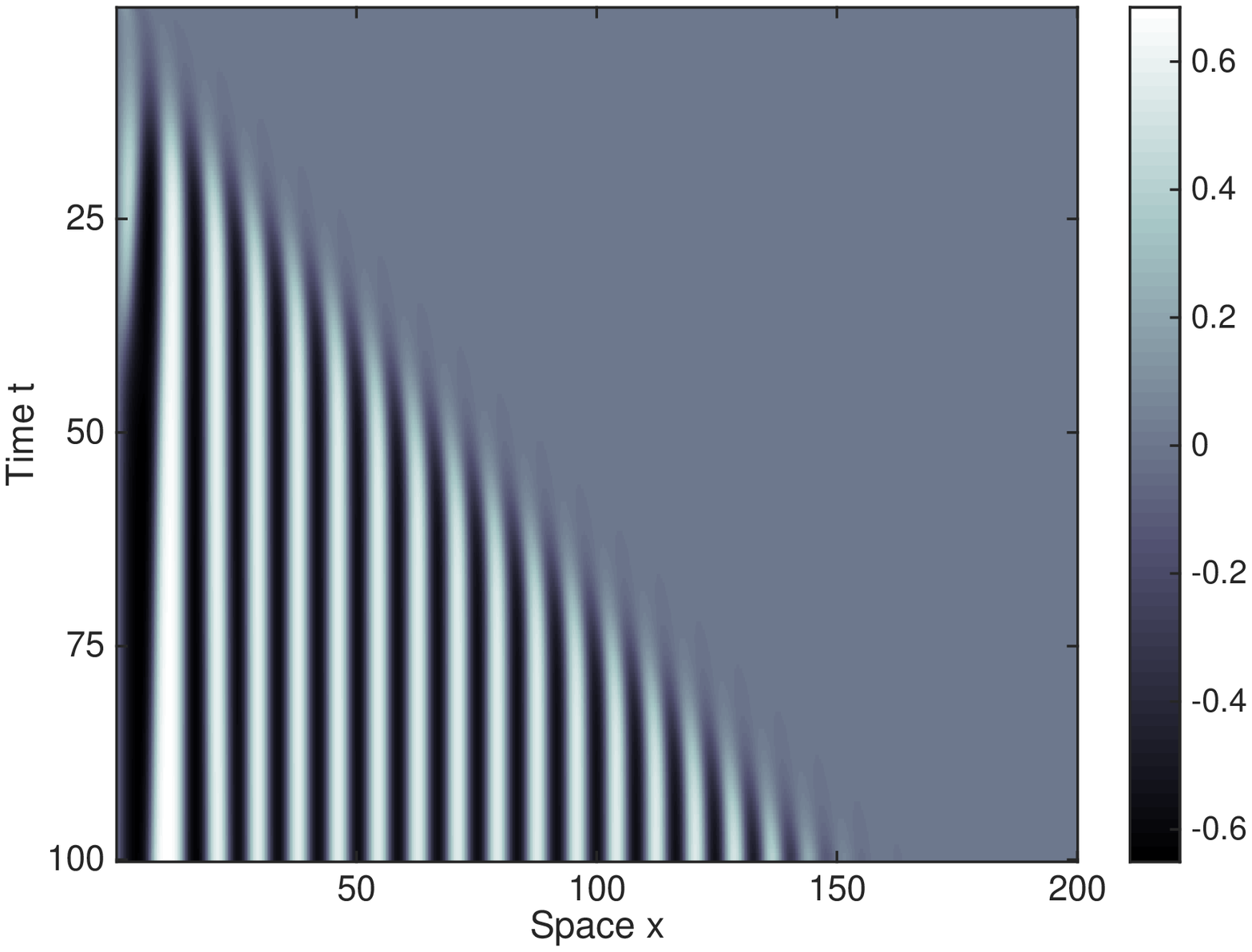}
\centering
\caption{Free invasion fronts in \eqref{e:CH} for $\gamma = 1.5$ (left) and $\gamma = -1.5$ (right). The invasion speed on the right is the linear speed predicted by the linearization about $u_*\equiv0$ while the invasion speed on the left is much faster and the corresponding front has a sharp leading edge, indicating a nonlinear front. The dashed red line overlaid on the left indicates the path of the pulled front on the right. Here \eqref{e:CH} was simulated using a semi-implicit time stepping method with second order finite differences in space ($dx = 0.2$) and first order in time ($dt = 0.01$). }
\label{f:chpp}
\end{figure}

\begin{figure}[h!]
\centering
\begin{subfigure}[h]{0.45\textwidth}
\includegraphics[width=\textwidth]{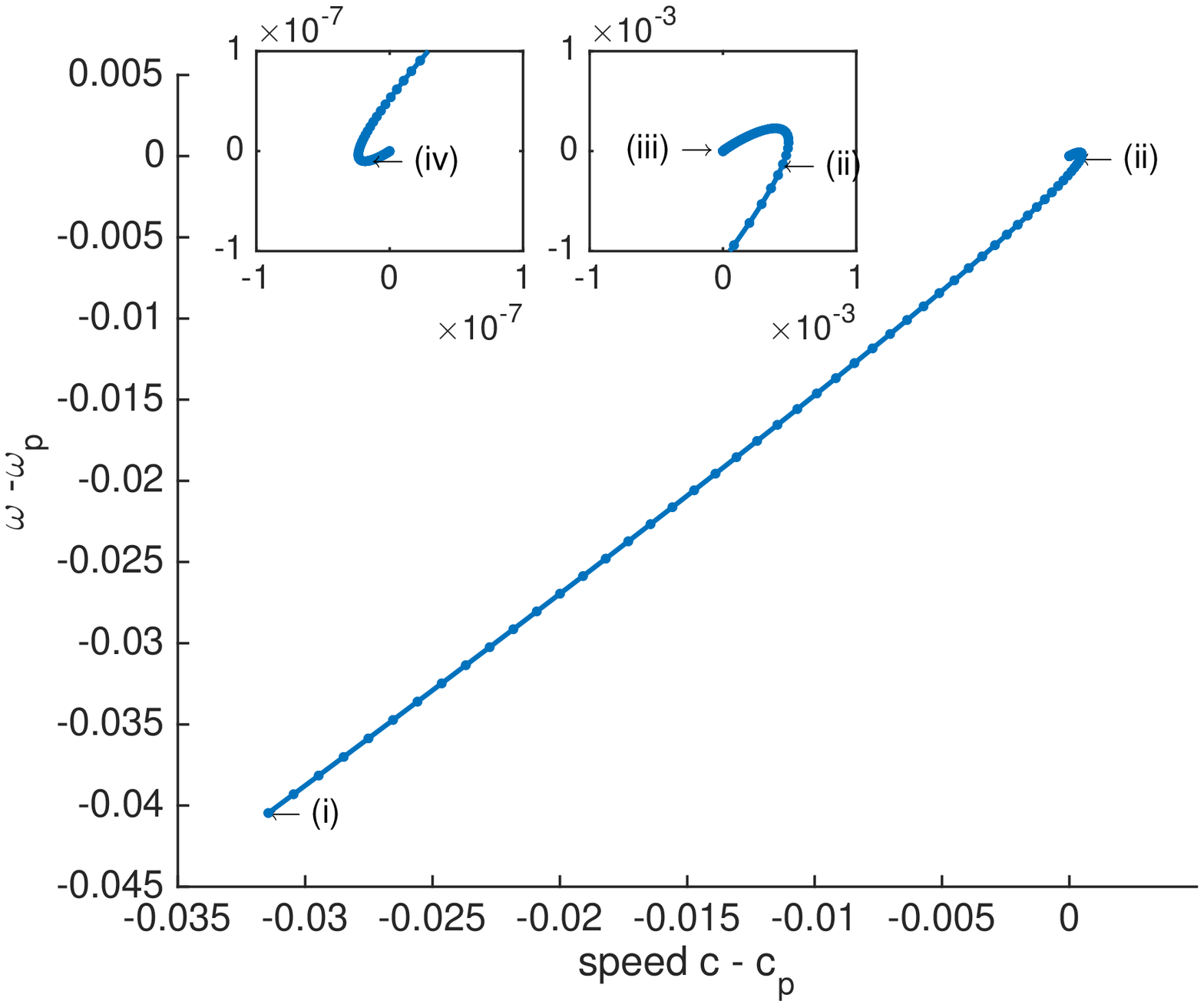}
\end{subfigure}
\hspace{-0.5in}
\begin{subfigure}[h]{0.45\textwidth}
\includegraphics[width=\textwidth]{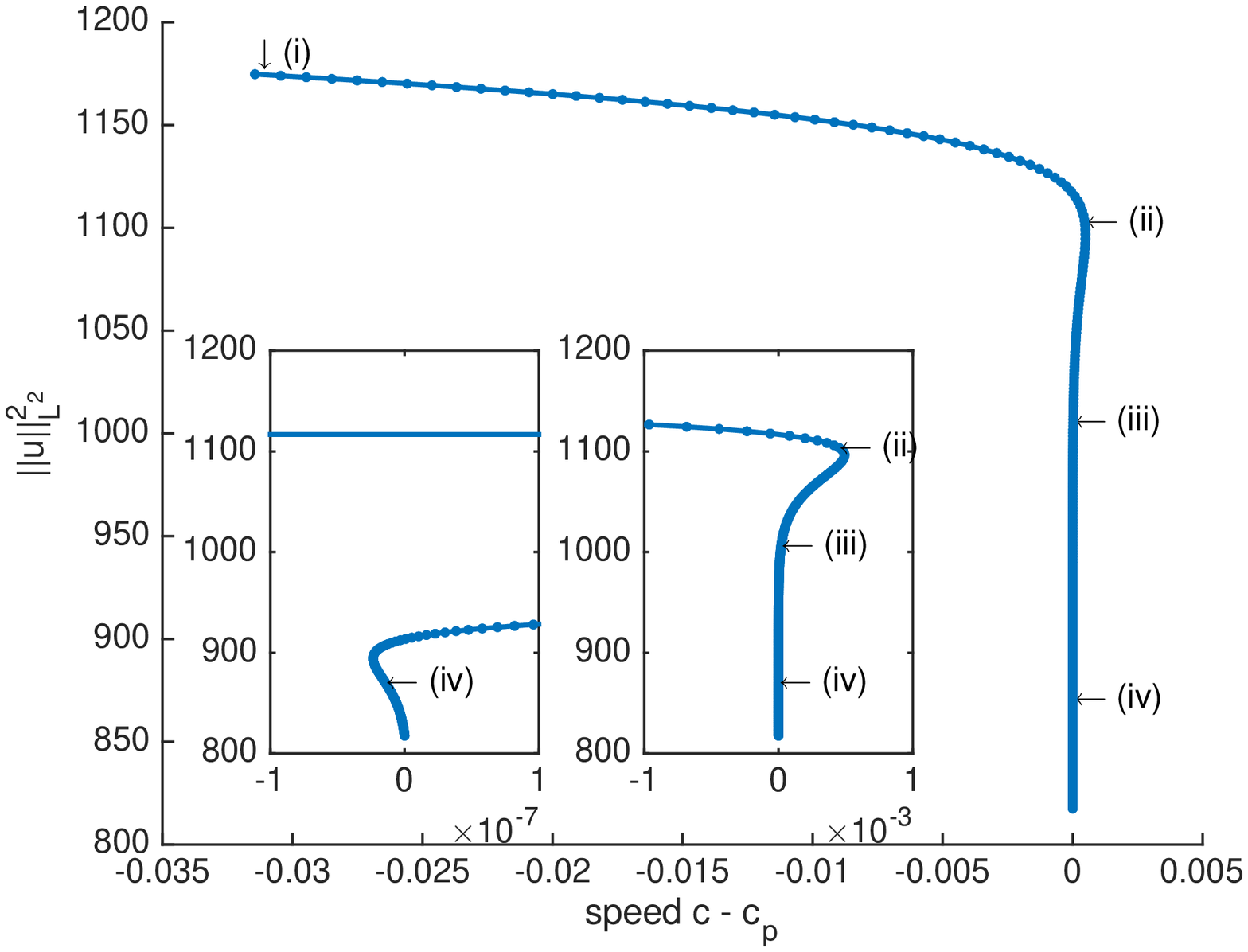}
\end{subfigure}

\begin{subfigure}[h]{1\textwidth}
\includegraphics[width=\textwidth]{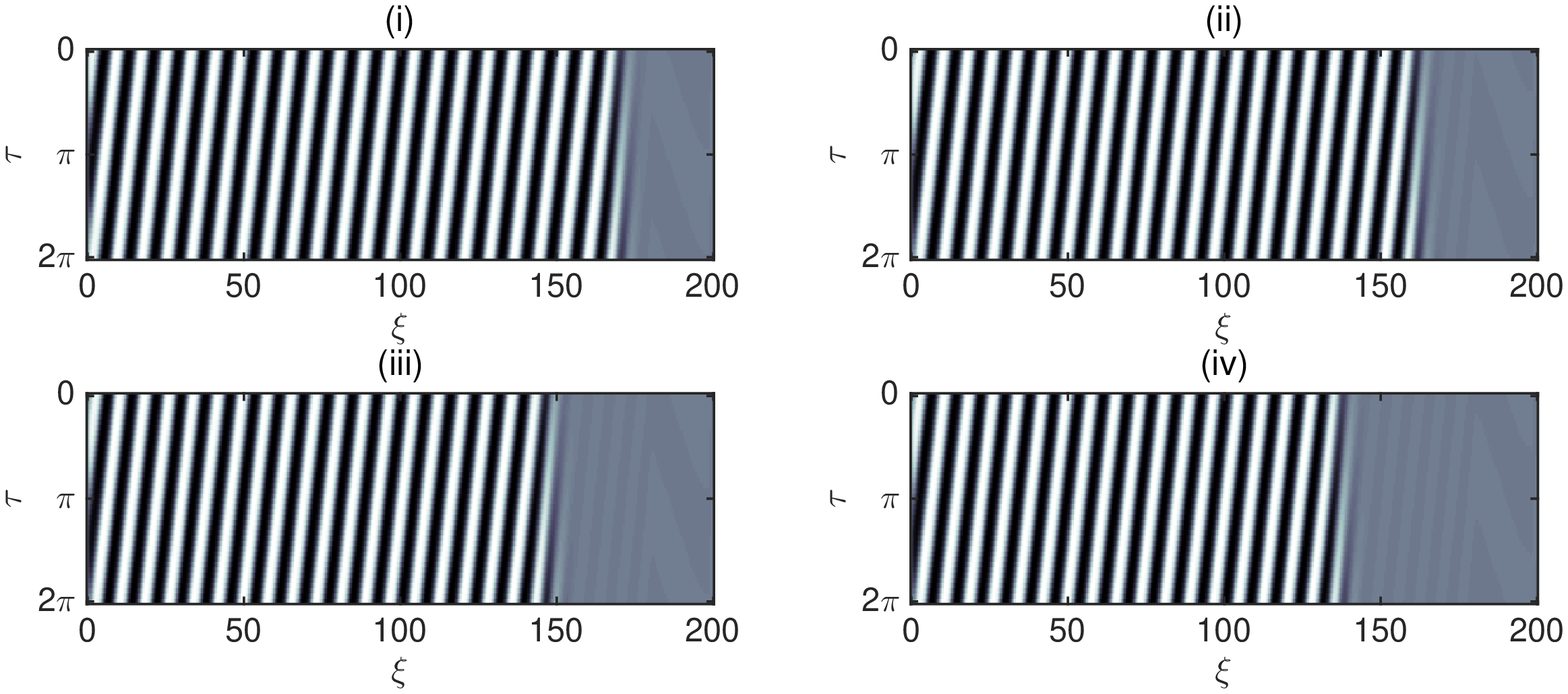}
\end{subfigure}
\vspace{-0.0in}
\caption{(upper left): Bifurcation curve for triggered pushed fronts in \eqref{e:mCH} with temporal frequency $\omega$ and trigger speed $c$ with $\gamma = 1.5$ for which the free pushed parameters are $(c_p,\omega_p) =(2.0324, 1.5115)$. (upper right): Plot of the $L^2$ norm of solutions against the trigger speed $c$. Insets are zoomed in near the value $c = c_\mathrm{p}$ (lower): Spactime diagrams of solutions for a selection of points (i): $(c,\omega) =(2.001,1.471)$,  (ii): $(c,\omega) =(2.0329,1.5113)$,   (iii): $(c,\omega) =(2.0325,1.5115)$,  (iv): $(c,\omega) =(2.0324,1.5115)$ along the bifurcation curve. First order forward differences for $\p_t$ and centered second-order differences for $\p_x$ were used, with step sizes $dt = 0.2, dx = 0.5$ respectively, and $N = 200$.  }
\label{f:chcont}
\end{figure}

Using numerical arc-length continuation we found that, in a narrow parameter regime, such fronts possess a spiraling bifurcation curve and thus exhibit locking and multi-stability phenomena, as in CGL.  Here we used the temporal frequency $\omega$ in our bifurcation diagrams and note that the wavenumber $k$ can be determined by the relation $c = \fr{\omega}{k}$ since the spatial pattern is stationary in a stationary frame. We also mention that this locking behavior was corroborated in semi-implicit time-stepping simulations.  In these simulations, if the homogeneous state was perturbed near the trigger, then the resulting patterned state would lock close to the trigger (i.e. farther out on the spiral).  If the homogeneous state was perturbed far away from the trigger then the pattern would lock far away from the trigger (i.e. closer to $(\omega_p,c_p)$ on the spiral).

Our numerical continuation method used finite-differences to discretize both temporal and spatial derivatives and the MATLAB Newton solver ``fsolve" to continue solutions on the domain $(\xi,\tau)\in [0,N]\times [0,2\pi]$ in $c$ and $\omega$ for some fixed $N$ large.  To accommodate the second parameter we appended the phase condition 
$$
\int_0^{2\pi} \la \p_\tau u(\cdot,s), u(\cdot,s) - u_\mathrm{old}(\cdot,s) \ra_{L^2([0,L])} ds = 1,
$$
where $u_\mathrm{old}$ is the solution found at the previous continuation step.  This eliminates the non-uniqueness due to the translation symmetry in time and allows the equation to be solved uniquely. The initial guess for the continuation algorithm was one full time-period of a solution obtained from an semi-implicit time-stepping method, with the same spatial discretizations as above.  More details of our continuation method can be found in the caption of Figure \ref{f:chcont}.

\section{Abstract formulation}\label{s:ab}
Our theoretical approach is motivated by the spatial dynamics method first formulated by Kirchgasser and subsequently developed by many others over the past few decades \cite{fiedler2003spatio,iooss1991bifurcating,kirchgassner1982wave, Sandstede01}.  By viewing a pattern-forming system as a continuous-time dynamical system, where the spatial variable is viewed as the ``time-like" variable,  the existence of a trigger front can be obtained as a heteroclinic bifurcation from a nearby pushed free front; see Figure \ref{f:gl} for a schematic of such a bifurcation.

In particular, we shall study a system of the form 
\begin{align}
\fr{d}{d\xi}u_0&= f_0(u_0;\mu)\label{e:sys1a}\\
\fr{d}{d\xi}u_1 &= Au_1 + f_1(u_0,u_1;\mu)u_1,\qquad \xi\in \R, \mu\in\R^2\label{e:sys1b},
\end{align}
where $\mu$ shall consist of system parameters, $u_0\in X_0:=\R^n$, and $u_1\in X_1$, a real Hilbert space. Equation \eqref{e:sys1a} governs the triggering mechanism as in \eqref{e:cgl3a}, while \eqref{e:sys1b} governs the pattern forming system.  Let $A:Y_1\subset X_1\rightarrow X_1$ be a closed linear operator where $Y_1:=D(A)$ is also a real Hilbert space which is dense and compactly embedded in $X_1$.  Furthermore we shall assume that there exists a projection $P$ such that $A^\rs:=AP$ and $-A^\ru := -(1-P)A$ are sectorial operators.  

Next we assume that the function $f_0$ satisfies
$$
f_0(0,\mu) =f_0(u_0^*(\mu),\mu) = 0,
$$
where $u_0^*(\mu)\in X_0$ varies smoothly in $\mu$ for all $\mu$ near the origin in $\R^2$. Furthermore we assume
$$
f_0\in \mc{C}^k( X_0\times \R^2, X_0),\qquad f_1\in \mc{C}^k(X_0\times X_1\times\R^2,X_1),
$$ for some $k\geq 1$. 
Defining $X:= X_0\times X_1$, it is readily seen that $U_* := (0,0),\,\, \tl U_*(\mu) :=(u_0^*(\mu),0)$ are equilibria of the system
\beq\label{e:sys2}
\fr{d}{d\xi} U = F(U;\mu),\qquad F(U;\mu) =  \left(\begin{array}{c}f_0(u_0;\mu) \\Au_1 + f_1(U;\mu)u_1\end{array}\right),\quad U = \left(\begin{array}{c}u_0 \\u_1\end{array}\right),
\eeq
 and 
 $$
 D_UF(\tl U_*(\mu);\mu) = \left(\begin{array}{cc}D_{u_0}f_0(u_0^*,\mu) & 0 \\ 0& A+f_1(u_0^*,0,\mu)\end{array}\right),\qquad
 D_UF( U_*;\mu) = \left(\begin{array}{cc}D_{u_0}f_0(0,\mu) & 0 \\ 0& A+f_1(0,0,\mu)\end{array}\right).
 $$

Next, motivated by the time-translation symmetry $\tau\rightarrow \tau - \theta$, which occurs in a typical spatial dynamics formulation, we assume the following 
\begin{Hypothesis}\label{h:sym}
Let $T_1: S^1\times X\rightarrow X$ be a strongly continuous group action of the circle, $S^1$, on $X$ such that $X_0\times\{0\}\subset \mathrm{Fix}(T_1)$, and $A$ and $f_1$ are both equivariant under this action.
\end{Hypothesis}

Finally, we assume the existence of a smooth family of periodic orbits.
\begin{Hypothesis}\label{h:po}
There exists a family of periodic solutions $U_p(\xi;\mu)$ of \eqref{e:sys2}, smooth in $\mu$,  which lie entirely in the subspace $\{0\}\times X_1$ and possess trivial isotropy with respect to $T_1$ which acts by,
$$
 T_1(k\zeta)\, U_p(\xi,\mu) = U_p(\xi+\zeta;\mu)
$$
where $k = k(\mu)$ defines the period, $2\pi/ k(\mu)$, of $U_p$. 
\end{Hypothesis}
In a typical spatial dynamics formulation, $k$ represents the spatial wavenumber of the periodic pattern.
 
 \begin{Remark}\label{r:q2}
 Note that the $u_1$-components of the equilibria $U_*$ and $\tl U_*$ are the same.  We have simplified the setting to reflect those in the examples given above where the trigger is a coefficient in the linear terms which progressively changes the stability of a constant preparation front.  We remark that our abstract setting could be readily altered to instead study a system with a source-term trigger which moves the system from one spatially homogeneous equilibrium to another.  
 \end{Remark}
\paragraph{CGL spatial dynamics}
In the setting of the complex Ginzburg-Landau equation given in \eqref{e:cgl3}-\eqref{e:cgl3a}, a formulation as above can be obtained by converting \eqref{e:cgl3} into a first order complex system  for $u$ and $v:=u_\xi$, and then decomposing into equations for the real and imaginary parts of each $u = s+\ri t, \,\,v = z+\ri w$ so that $s_\xi = z$ and $t_\xi = w$.  Setting $u_0 = \chi_\epsilon$ and $u_1 = (s,t,z,w)^T$, one obtains the system
\begin{align}
\epsilon \fr{d}{d\xi}u_0 &= u_0^2 - 1,\\
\fr{d}{d\xi}u_1 &=A(c,\omega) u_1 + f_1(u_0,u_1;c,\omega) u_1 \label{e:CGLsd}  
\end{align}
with 
\begin{align}
&A(\alpha,c,\omega) = -\fr{1}{1+\alpha^2}\left(\begin{array}{cc}0 & I_2 \\A_1 & A_2\end{array}\right),\qquad
 f_1(u_0,u_1;c,\omega)u_1 =Bu_1+ c(u_1)^T Cu_1 + d(u_1)^TDu_1,\notag \\
B&= -\fr{1}{1+\alpha^2}\left(\begin{array}{cc}0 &0\\B_1 & 0\end{array}\right),\quad
C = -\fr{1}{1+\alpha^2} \left(\begin{array}{cc}0 &0\\C_1 & 0\end{array}\right),\quad
D= \fr{1}{1+\alpha^2} \left(\begin{array}{cc}0 &0\\D_1 & 0\end{array}\right),\notag
\end{align}
where all the zeros are $2\times2$ zero-matrices, $I_2$ is the $2\times2$ identity, and 
\begin{align}
A_1 &= \left(\begin{array}{cc}1-\alpha\omega & \omega+\alpha \\-(\omega+\alpha) & 1-\alpha\omega\end{array}\right), \;
B_1 = \left(\begin{array}{cc}u_0-1 & \alpha(u_0-1) \\\alpha(1-u_0) & u_0-1\end{array}\right),\;
C_1 =  \left(\begin{array}{cc}\rho+\alpha\gamma & \alpha\rho-\gamma \\\gamma-\alpha\rho & \rho+\alpha\gamma\end{array}\right),\notag\\
D_1 &= \left(\begin{array}{cc}1+\alpha\beta & \alpha-\beta \\\beta-\alpha & 1+\alpha\beta\end{array}\right),\quad
c(u_1) = (s^2+t^2)\cdot (0,0,1,1)^T,\quad
d(u_1) = (s^2+t^2)^2\cdot (0,0,1,1)^T.\notag
\end{align}
Here, the phase space is simply $X = \R^5$.  The preparation front $u_\rpr$ corresponds to a heteroclinic orbit contained in the $\{u_1=0\}$ subspace, while $u_\mathrm{ff}$ is a heteroclinic orbit contained in the $\{u_0 = 1\}$ subspace. Also the $S^1$-action arises as the gauge-symmetry,
$$
T_1(\theta): (u, v)\mapsto \re^{\ri \theta}( u, v).
$$ 

\paragraph{Cahn-Hilliard spatial dynamics}
In the context of the modified Cahn-Hilliard equation given in \eqref{e:mCH}, a formulation as above can be obtained by setting 
$$
u_0 := \chi,\quad u_1 = (u,v,\theta,w)^T :=(u,\,u_\xi,\, u_{\xi\xi} + \tl f(\xi,u),\, (u_{\xi\xi} + \tl f(\xi,u))_{\xi})^T
$$
from which one finds
\begin{align}
\epsilon \fr{d}{d\xi}u_0 &= u_0^2 - 1,\\
\fr{d}{d\xi}u_1 &=A(c,\omega) u_1 + f_1(u_0,u_1;\gamma) u_1, \label{e:CHsd}  
\end{align}
with
\beq
A(c,\omega)  = \left(\begin{array}{cc}b_1& I_3 \\-\omega \p_\tau  & b_2(c)\end{array}\right),\quad
f_1(u_0,u_1;\mu)  = -\left(\begin{array}{cccc}0 & 0 & 0 & 0 \\1-u_0 +\gamma u^2 - u^4 & 0 & 0 & 0 \\0 & 0 & 0 & 0 \\0 & 0 & 0 & 0\end{array}\right),  
\eeq
where $0<\epsilon\ll1$, $I_3$ is the three dimensional identity matrix, $b_1 = (0,1,0)^T,$ and $b_2(c) = (c,0,0)$. This is then an ill-posed evolution equation on the Banach space $X = \R\times H^3(\mb{T})\times  H^2(\mb{T})\times H^1(\mb{T})\times  L^2(\mb{T})$, where the linear operator $A$ has domain $Y = \R\times H^4(\mb{T})\times  H^3(\mb{T})\times H^2(\mb{T})\times  H^1(\mb{T})$.  By setting $\mu = (c - c_\mathrm{p},\omega - \omega_\mathrm{p})$, we obtain a system of the form given in \eqref{e:sys2} above.  In this form, $u_\rpr$ corresponds to a heteroclinic orbit contained in the $\{u_1=0\}$ subspace, while $u_\mathrm{ff}$ is a heteroclinic orbit contained in the $\{u_0 = 1\}$ subspace.  Here, the $S^1$-action arises as a time-shift symmetry $\tau\mapsto \tau-\theta$.

\subsection{Spectral hypotheses}\label{ss:spechyp}
Next, we state our spectral hypotheses for the relative equilibria $U_*, \tl U_*,$ and $U_p$ of \eqref{e:sys2}. It follows from the compact embedding of $Y_1\subset X_1$ that the spectra of $D_UF$, evaluated at each of these, consists of isolated eigenvalues of finite multiplicity. We thus assume the following,

\begin{Hypothesis}\label{h:eig1}

\begin{enumerate}
\item The linearization of $F$ about $U_*$ at $\mu =  0$ has the following properties:
\begin{itemize}
\item The operator $D_UF(U_*;0)$ has algebraically simple eigenvalues $\nu_\rss = -r_\rss \pm \ri \sigma_\rss, \,\,\nu_\rsu = -r_\rsu\pm \ri \sigma_\rsu$ such that $r_\rss> r_\rsu>0,\, \sigma_\rss,\sigma_\rsu\neq0$, and all other $\nu\in \Sigma\lp(D_UF(U_*;\mu)\rp)$ satisfy either $\rre\{\nu\}>-r_\rsu$ or $\rre\{\nu\}<-r_\rss$. 
\item $D_{u_0}f_0(0,0)$ has a real unstable eigenvalue $\nu_\ru  = r_\ru >0$ which satisfies $r_\ru>2r_\rss - r_\rsu$. 
\end{itemize}
\item The periodic orbit $U_p$ is hyperbolic.  That is the linearization about $U_p$ has spectrum bounded away from the imaginary axis except for a simple Floquet exponent, located at $0\in \C$. 
\item  The spectrum $\Sigma(D_UF(\tl U_*);0)$ is bounded away from the imaginary axis. That is, there exists a $\gamma>0$ such that all eigenvalues satisfy $|\rre\{\nu\}|>\gamma$.
 \end{enumerate}

\end{Hypothesis}

\begin{figure}[h!]\label{f:not}
\centering
\includegraphics[width=0.43\textwidth]{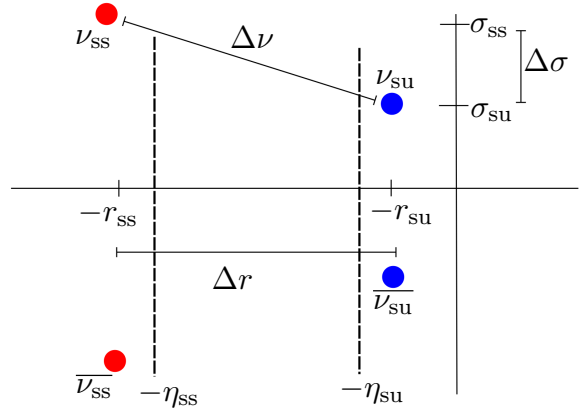}
\caption{Schematic diagram of notation for leading eigenvalues and relevant quantities.}
\end{figure}
Hypothesis \ref{h:eig1}(i) encodes the spectral splitting corresponding to the leading order decay of the free pushed front and also describes the decay of the preparation front in backwards time, requiring that it decays with a fast rate. This will aid in our analysis and is not restrictive in our results since we imagine such a front to be controlled by the experimenter or an outside mechanism.  For instance, for the examples given in Section \ref{s:ex}, one could obtain such a fast decay by tuning $\epsilon>0$ to be sufficiently small.  Hypothesis \ref{h:eig1}(ii) readily gives that $U_p$ is not degenerate with respect to perturbations in $\mu$. Hypotheses \ref{h:eig1}(iii) reflects the fact that the state $\tl U_*$, which corresponds to the asymptotic state ahead of the trigger, is typically PDE-stable. 

\begin{Remark}\label{r:cc}
These spectral hypotheses would need to be adapted if the PDE from which the system originated possessed any conserved quantities or additional symmetries.  We briefly discuss how our results would change for these cases in Section \ref{s:dis}. 
\end{Remark}

Using the sectoriality of the decomposition of $A$, we can define spectral projections $P_{1,\infty}^{\rss/\rsu}$ to obtain eigenspaces, $E_{1,\infty}^{\rss}$ and $E_{1,\infty}^{\rsu}$ of $D_UF(U_*)$ which are associated with the spectral splitting  in Hypothesis \ref{h:eig1}(i). These spaces have the decomposition
$$
E_{1,\infty}^{\rss/\rsu} = E_{1,\infty}^{\rss/\rsu,\rl} + E_{1,\infty}^{\rss/\rsu,\rs/\ru},
$$
where ``$\rl $" denotes the 2-dimensional eigenspaces corresponding to the leading eigenvalues $\nu_{\rss/\rsu}$ and ``$\rs/\ru$" denote the eigenspaces corresponding to the spectral sets $\{\nu<-r_\rss\}$ and $\{\nu>-r_\rsu\}$ respectively.  Also, let $e_{1,\infty}^{\rss/\rsu}\in  E_{1,\infty}^{\rss/\rsu,\rl}$ denote the unit-normed complex eigenvectors of $D_UF(U_*)$ associated with the eigenvalues $\nu_{\rss/\rsu}$, and let $e_{j,\infty}^*$ denote the complex eigenvector of the adjoint linearization $-D_UF(U_*)^*$ with eigenvalue $-\overline{\nu_\rsu}$.

From these spectral hypotheses we have the following result on locally invariant manifolds around $U_*$.
\begin{Lemma}\label{l:locmfld}
According to the spectral splitting 
$$
\Bigg(\Sigma\lp(D_UF(U_*;0)\rp)\cap \lp\{\rre\,\,\nu\leq -r_\rss\rp\}\Bigg)\,\,\bigcup \,\,\Bigg(\Sigma\lp(D_UF(U_*;0)\rp)\cap \lp\{\rre\,\,\nu\geq -r_\rsu\rp\}\Bigg),
$$
 the system \eqref{e:sys2} possesses locally-invariant manifolds $W^\rss_\mathrm{loc}(U_*)$ and $W^\rsu_\mathrm{loc}(U_*)$ which are $C^k$- and $C^1$-smooth respectively.  Furthermore, the periodic orbit $U_p$ and equilibrium $\tl U_*$ possess $C^k$-smooth locally invariant manifolds $W^\rcu_\mathrm{loc}(U_p)$, and $W^\rs_\mathrm{loc}(\tl U_*)$.
\end{Lemma}
\begin{proof}
This follows by standard results on infinite dimensional locally invariant manifolds (see \cite{henry81} or \cite{Iooss92}). We also mention that higher degrees of smoothness of $W^\rsu_\mathrm{loc}(U_*)$ can be obtained if the spectral gap $\Delta\eta = r_\rss - r_\rsu$ is sufficiently large.
\end{proof}

Next we state our assumptions on the heteroclinic orbits formed by the preparation and pushed free fronts.
\begin{Hypothesis}\label{h:fa}
For $\mu = 0$, there exist $C^k$-smooth heteroclinic solutions $q_i^0(\xi)$ of \eqref{e:sys2} for $i = 1,2$ such that for some $S>0$ sufficiently large
\begin{itemize}
\item $q_1^0(\xi)\in \{\{0\}\times X_1\}$ and $q_2(\xi)\in\{X_0\times\{0\}\}$ for all $\xi\in \R$,
\item  $\{q_1^0(\xi)\}_{\xi\in [S,\infty)} \subset W^\rss_\mathrm{loc}(U_*)$,
\item  $\{q_1^0(\xi)\}_{\xi\in (-\infty,-S]} \subset W^\rcu_\mathrm{loc} (U_p)$,
\item $\{q_2^0(\xi)\}_{\xi\in (-\infty,-S]} \subset W^\rsu_\mathrm{loc}(U_*)$,
\item $\{q_2^0(\xi)\}_{\xi\in [S,\infty)} \subset W^\rs_\mathrm{loc}(\tl U_*)$.
\end{itemize}

Furthermore, for some $\epsilon>0$ small, there exist $a, b\in \C$ such that $q_i^0$ has the following asymptotics,
\begin{align}
q_1^0(\xi)  &=  a \re^{\nu_\rss \xi}e_{1,\infty}^{\rss,\rl} + \mathrm{c.c.} + \mc{O}(\re^{-(r_\rss + \epsilon) \xi}), \quad\text{as}\quad \xi\rightarrow +\infty,\notag\\
q_2^0(\xi)  &=  b \re^{\nu_\ru \xi}e_{1,\infty}^{\ru} +\mc{O}(\re^{( r_\ru+ \epsilon) \xi}), \quad\text{as}\quad \xi\rightarrow -\infty,\notag
\end{align}
 where $\mathrm{c.c.}$ stands for complex-conjugate and the vectors $e_{1,\infty}^{\rss,\rl}\in E_{1,\infty}^{\rss,\rl}$, and $e_{1,\infty}^\ru\in E_{1,\infty}^{\ru}$ have unit-norm and are complex eigenvectors of the linearization $D_UF(U_*)$ associated with the leading eigenvalues $\nu_\rss$ and $\nu_\ru$ respectively. 
 
Finally, the orbit $q_2^0$ is robust to perturbations in $\mu$.  That is there exists a smooth family of heteroclinic orbits $q_2(\xi,\mu)$ satisfying the above properties for all $|\mu|\ll 1$ and $q_2(\xi;0) = q_2^0(\xi)$.  
\end{Hypothesis}

The second part of this hypothesis states that, to leading order, $q_1^0$ and $q_2^0$ approach $U_*$ along the leading eigenspaces $E^{\rss, \rl}_{1,\infty}$ and $E^{\ru,\rl}_{1,\infty}$ as $\xi\rightarrow \pm \infty$ respectively. In this notation, $q_1^0$ denotes the pushed free front, while $q_2$ denotes the preparation front.

\subsection{Invariant manifolds and variational set-up}\label{ss:invvar}
We now construct global invariant manifolds in neighborhoods of the heteroclinic orbits $q_1^0$ and $q_2^0$.  To do so we define variations, $w_i^0(\xi) =  U(\xi)- q_i^0(\xi)$, about such orbits with $i = 1,2$, and study the variational equations
\beq\label{e:var0}
\fr{d}{d\xi} w_i = A_i(\xi) w_i^0 + g_i^0(\xi,w_i),\qquad \xi\in \R,
\eeq
with
$$
A_i(\xi) := D_UF(q_i^0(\xi);0), \quad g_i^0(\xi,w_i) := F(q_i^0(\xi) + w_i; \mu) - F(q_i^0(\xi);0) - A_i(\xi) w_i^0.
$$ 
In order to study these variations we shall use exponential dichotomies of the linear variational equations and their adjoints 
\begin{align}
\fr{d}{d\xi} w &= A_i(\xi) w,\label{e:expd}\\
\fr{d}{d\xi} \psi &= -A_i(\xi)^* \psi,\qquad i = 1,2.\label{e:adjvar}
\end{align}
Before doing so, we require the following well-posedness assumption.
\begin{Hypothesis}\label{h:ex}
For both $i = 1,2$, if $w_0(\xi)$ is a bounded solution of either of the linear variational equations \eqref{e:expd} or \eqref{e:adjvar}
 for all $\xi\in \R$ and $w_0(\xi_0) = 0$ for some $\xi_0 \in \R$, then $w_0\equiv 0$.
\end{Hypothesis}
We remark that for finite-dimensional systems and many parabolic equations this hypothesis holds via parabolic regularity results \cite{chen1998strong,Sandstede01}.

\begin{Proposition}{(Existence of Exponential Dichotomies)}\label{p:expdic}
Assuming the above hypotheses, \eqref{e:expd} has exponential dichotomies on $J_1 = \R^+,\, J_2 = \R^-$ with a splitting according to the eigenspaces $E_{1,\infty}^{\rss/\rsu}$ given above. That is there exist projections $P_i^{\rss/\rsu}(\xi):X\rightarrow X$ for $\xi\in J_i$ such that the following  holds for some $K>0$:
\begin{itemize}
\item
For any $\zeta\in J_i$ and $u\in X$, there exists a solution $\Phi_i^\rss(\xi,\zeta) u$ of \eqref{e:expd} defined for $\xi\geq \zeta$, continuous in $(\xi, \zeta)$ for $\xi\geq \zeta$, and differentiable in $(\xi,\zeta)$ for $\xi>\zeta$, such that $\Phi_i^\rss(\zeta,\zeta) u = P_i^\rss(\zeta) u$ and
\beq
|\Phi_i^\rss(\xi,\zeta) u| \leq K \,\re^{-r_\rss (\xi - \zeta)}|u|, \quad \xi\geq \zeta\
\eeq
\item 
For any $\zeta\in J_i$ and $u\in X$, there exists a solution $\Phi_i^\rsu(\xi,\zeta) u$ of \eqref{e:expd} defined for $\xi\leq \zeta$, continuous in $(\xi, \zeta)$ for $\xi\leq \zeta$, and differentiable in $(\xi,\zeta)$ for $\xi<\zeta$, such that $\Phi_i^\rsu(\zeta,\zeta) u = P_i^\rsu(\zeta) u$ and
\beq
|\Phi_i^\rsu(\xi,\zeta) u| \leq K \,\re^{-r_\rsu(\xi - \zeta)}|u|, \quad \xi\leq \zeta.
\eeq
\item The solutions $\Phi_i^\rss(\xi,\zeta) u$ and $\Phi_i^\rsu(\xi,\zeta) u$ satisfy
\begin{align}
\Phi_i^\rss(\xi,\zeta) u &\in R(P_i^\rss(\xi)) \quad \text{for all}\qquad \xi\geq \zeta, \qquad  \xi,\zeta \in J_i,\notag\\
\Phi_i^\rsu(\xi,\zeta) u & \in R(P_i^\rsu(\xi)) \quad \text{for all} \qquad \xi \leq \zeta,\qquad \xi,\zeta \in J_i,\notag
\end{align}
\end{itemize}
where $|\cdot|$, unless otherwise stated, denotes the norm on $X$.
\end{Proposition}
\begin{proof}
See \cite{Peterhof97} or \cite{Sandstede01}.
\end{proof}
The first two bullets of this proposition correspond to the usual stable-unstable dichotomy when considered in an weighted norm $||u||_{\eta}:=\sup_{\xi\in J_i} \re^{\eta \xi}|u(\xi)|$ with $0<r_\rsu < \eta< r_\rss$.   

Let us denote $E_i^j(\xi) = P_i^j(\xi)X$, for $i = 1,2$ and $j = \rss,\rsu$.  Also, as they will be necessary to the subsequent analysis, we isolate the leading components of these $\xi$-dependent subspaces as follows
\beq\label{e:spl}
E_{i}^\rss(\xi) = E_{i}^{\rss,\mathrm{l}}(\xi) +  E_{i}^{\rss,\rs}(\xi),\qquad E_{i}^\rsu(\xi) = E_{i}^{\rsu,\mathrm{l}}(\xi) +  E_{i}^{\rsu,\ru}(\xi), \quad i = 1,2,
\eeq
such that the spaces $E_{i}^{\rss/\rsu,\mathrm{l}}(\xi) $ are unique and satisfy
\begin{align}
E_{1}^{\rss/\rsu,\mathrm{l}}(\xi) &\rightarrow E_{1,\infty}^{\rss/\rsu,\mathrm{l}}\qquad \xi\rightarrow \infty ,\notag\\
E_{2}^{\rss/\rsu,\rs/\ru}(\xi) &\rightarrow E_{1,\infty}^{\rss/\rsu,\rs/\ru}\qquad \xi\rightarrow -\infty.
\end{align}
Such a decomposition of, say for example $E^{\rss}_2(\xi)$, can be achieved by first obtaining exponential dichotomies associated with the spectral sets $\{\nu\,:\,\rre\{\nu\} \leq -r_\rss\}, \, \{\nu\,:\,\rre\{\nu\} \geq -r_\rss\}$ and then taking the intersection of their associated $\xi$-dependent subspaces.  Denote the resulting dichotomies of these restricted subspaces as $\Phi^{\rss/\rsu,\rl}_{i}$ for $i = 1,2$.

From Proposition \ref{p:expdic}, we are then able to obtain the existence of globally invariant manifolds in a neighborhood of $q_i^0(\xi)$ for all $\xi\in J_i$, and $\mu$ sufficiently small.

\begin{Proposition}\label{p:mfld}
For all $\mu$ sufficiently small the equilibrium \eqref{e:sys2} possesses strong stable and weak-stable/unstable invariant manifolds $W^\rss(U_*)$ and $W^\rsu(U_*)$ which exist in a neighborhood of the orbits $q_1^0$ and $q_2^0$ respectively.  Furthermore, $W^\rss(U_*)$ is $C^k$-smooth while $W^\rsu(U_*)$ is in general only $C^1$-smooth.  In an exponentially weighted space with weight $\eta\in(r_\rsu,r_\rss)$, $W^\rss(U_*)$ contains all solutions which stay close to $q_1^0(\xi)$ for all $\xi\geq 0$ while $W^\rsu(U_*)$ contains all solutions which stay close to $q_2^0(\xi)$ for all $\xi\leq 0$. Finally, these manifolds are smooth in the parameter $\mu$ and have tangent spaces which satisfy, for $\mu = 0$,
$$
T_{q_i^0(\xi)} W^{j}(U_*) = E^j_i(\xi),\qquad j = \rss,\rsu,\,\, i = 1,2.
$$
\end{Proposition}
\begin{proof}
This proof follows in the same way as those in \cite[Sec. 3.5]{Sandstede99} which use the existence of exponential dichotomies from \cite[Thm. 3.3.3]{Peterhof97} and infinite-dimensional center manifold results of \cite{Iooss92}; see also \cite{Sandstede01,sandstede2004defects}.
\end{proof}
In an analogous fashion, one may use the spectral properties of the linearization about $U_p$ and $\tl U$ to obtain the following proposition,
\begin{Proposition}\label{p:mfld1}
For all $\mu$ sufficiently small, the equilibria $U_p$ and $\tl U_*$ of \eqref{e:sys2} possess $C^k$-smooth center-unstable and stable manifolds, denoted as $W^\rcu(U_p)$ and $W^\rs(\tl U_*)$, which exist in a neighborhood of the orbits $q_1^0$ and $q_2^0$, and are smooth in the parameter $\mu$.   Here, $W^\rcu(U_p)$ contains all solutions which stay close to $q_1^0(\xi)$ for all $\xi\leq 0$ and $W^\rs(\tl U_*)$ contains those which stay close to $q_2^0(\xi)$ for all $\xi\geq 0$.  \end{Proposition}
\begin{proof}
The hypothesis on the linearizations at $U_p$ and $\tl U_*$ give the existence of center-unstable and stable dichotomies $\Phi_{-1}^{\rs/\rcu}$ along $q_1^0(\xi)$ for $\xi\leq0$, and stable and unstable dichotomies $\Phi_{+2}^{\rs/\ru}$ along $q_2^0(\xi)$ for $\xi\geq0$. As in Proposition \ref{p:mfld}, one can then use these dichotomies and variation of constants formulas to prove the above proposition.
\end{proof}

\subsection{Intersection hypotheses}\label{ss:int}
We wish to construct pushed trigger fronts as intersections between $W^\rcu(U_p)$ and $W^\rs(\tl U_*)$ near the equilibrium $U_*$ under certain conditions on the heteroclinic chain composed of $q_1^0$ and $q_2^0$. We first assume the tangent spaces of the invariant manifolds along $q_i^0(\xi)$ generically behave as a codimension-two heteroclinic bifurcation problem: 
\begin{Hypothesis}\label{h:tan}
\begin{enumerate}
\item The tangent spaces $T_{q_1^0(0)}W^\rcu(U_p),$ and $T_{q_1^0(0)}W^\rss(U_*)$ form a Fredholm pair with index 0, and satisfy
$$
\dim \lp(T_{q_1^0(0)}W^\rcu(U_p) +  T_{q_1^0(0)}W^\rss(U_*)\rp)^\perp = \dim \lp(T_{q_1^0(0)}W^\rcu(U_p) \cap  T_{q_1^0(0)}W^\rss(U_*)\rp) =  2.
$$ 
\item The tangent spaces $T_{ q_2^0(0)}W^\rsu(U_*),$ and $T_{q_2^0(0)} W^\rs(\tl U_*)$ intersect transversely, form a Fredholm pair of index 1, and satisfy 
$$
\dim\lp(T_{ q_2(0)}W^\rsu(U_*)\cap T_{ q_2(0)} W^\rs(\tl U_*)\rp)  = 1.
$$
\end{enumerate}
\end{Hypothesis}
These hypotheses enforce genericity on the heteroclinic orbits in the sense that $q_1^0$ can be transversely unfolded in the parameter $\mu$. In the setting of an evolutionary PDE with both time- and space-translational symmetries (like the Cahn-Hilliard equation mentioned above), $q_1^0$ is a modulated traveling wave with both $\tau$- and $\xi$-derivative lying in the intersection $T_{q_1^0(0)}W^\rcu(U_p) \cap  T_{q_1^0(0)}W^\rss(U_*)$, while $q_2^0$ lies in the subspace of time-independent functions and thus has only $\xi$-derivative lying in the intersection $T_{ q_2(0)}W^\rsu(U_*)\cap T_{ q_2(0)} W^\rs(\tl U_*)$.  See \cite[Sec. 4.3]{Goh11} for more discussion on this topic.  

We must also make an assumption on the inclination properties of the invariant manifolds between $U_*$ and $\tl U_*$. Let $P_{2,+}^\rsu(\xi)$ denote the projection in $X$ onto $E^\rsu_2(\xi)=T_{q_2^0(\xi)}W^\rsu(U_*)$ along $\tl E_{+2}^\rs(\xi)$, the orthogonal complement of $\fr{dq_2^0}{d\xi}(\xi)$ in $T_{q_2^0(\xi)}W^\rs(\tl U_*)$.  Such a projection can be constructed in the same manner as in \cite[Eqn. 3.20]{Peterhof97}. 
\begin{Hypothesis}(Inclination property)\label{h:inc}
The restricted projection 
$$
\mc{P}^\rl:= P_{2,+}^\rsu(0)\,\Big|_{E_2^{\rss,\rl}(0)}
$$
is an isomorphism from $E_2^{\rss,\rl}(0)$ onto $E_2^{\rsu,\rl}(0)$.
\end{Hypothesis}
We note that the equivariance of $F$ with respect to the $S^1$-action implies that $\mc{P}^\rl$ commutes with $T_1$.
This equivariance makes $\mc{P}^\rl$ complex linear when considered on the complexification of the subspaces $E^{\rss/\rsu,\rl}_2(0)$ and will enforce certain conditions on the coefficients of the bifurcation equation, see Section \ref{ss:exp} below.  Additionally, this hypothesis can be given a geometric interpretation when the invariant manifold $W^\rs(\tl U_*)$ about $q_2(\xi)$ can be extended for all $\xi\in \R$, as is the case when $X$ is finite dimensional.  In such a situation, this hypothesis says that $W^\rs(\tl U_*)$ converges towards the non-leading strong-stable eigenspace $E_{1,\infty}^{\rss,\rs}$ in backwards time and hence does not lie in an inclination-flip configuration \cite{homburg2010homoclinic}. 

\begin{Remark}\label{r:fsym}
Since $X$ is a Hilbert space, it can be decomposed as a sum of complex one-dimensional irreducible representations.  In the spatial-dynamics formulation for the Cahn-Hilliard equation for $u_0  \equiv 1$, this decomposition is simply the Fourier series
$$
U(t) = \sum_{\ell \in \Z} U_\ell \re^{\ri \ell \tau}, \quad U_\ell \in \R^4,
$$
and can be used to determine the spatial eigenvalues of the linearization of \eqref{e:CHsd} about the equilibrium $u_1 = 0$.  Replacing $\p_\tau$ by $\ri \ell$, and setting $(c,\omega) = (c_\mathrm{p}, \omega_\mathrm{p})$, the linearization can be broken down into a set of infinitely many finite-dimensional linear systems, whose eigenvalues $\nu_\ell$ satisfy
\beq\label{e:CHeig}
0=\nu_\ell^4 + f'(u_*)\nu_\ell^2 - c_\mathrm{p} \nu_\ell + \ri\omega_\mathrm{p}\ell, \quad \ell\in \Z,
\eeq
with corresponding eigenspaces lying in the subspaces
$$
Y_\ell   = \mathrm{span}_{U_\ell,U_{-\ell}\in \R^4}\{ U_\ell \re^{\ri\ell \tau},U_{-\ell} \re^{-\ri\ell \tau}\}.
$$
Hypothesis \ref{h:inc} then requires that each of the leading eigenspaces, $E_{1,\infty}^{\rss/\rsu,\rl}$, must lie in the same subspace $Y_\ell$ for some $\ell$.  If this was not true, we would obtain that the two irreducible representations $\theta\mapsto \re^{\ri \ell_{1}\theta}$,  and $\theta\mapsto \re^{\ri \ell_{2}\theta}$ for distinct $\ell_1$ and $\ell_2$ are isomorphic, a contradiction. 

\end{Remark}

As in other heteroclinic bifurcation problems, we must require the invertibility of a certain mapping constructed using Melnikov integrals. Hence for $j = 1,2$ we let $e_j^*(\xi)$ be bounded solutions of the adjoint variational equation \eqref{e:adjvar} such that $e_j^*(0) = e_{j,0}^*$ for vectors $e_{j,0}^*\in X$ with unit-norm which satisfy
$$
\text{span}_{j=1,2} \{ e_{j,0}^*\}=  \lp(T_{q_1(0)}W^\rcu(U_p) +  T_{q_1(0)}W^\rss(U_*)\rp)^\perp.
$$
We then assume the following 
\begin{Hypothesis}\label{h:mel}

The following mapping is invertible,
\begin{align}
\mc{M}:& \R^2\rightarrow (E_1^{\rss}(0)+E_{-1}^\rcu(0))^\perp,\notag\\
&\, \mu\mapsto \sum_{i = 1,2} \int_{-\infty}^\infty\langle D_\mu F(q_1^0(\zeta);0) \mu, e^*_j(\zeta) \rangle d\zeta \,\,e_{j,0}^*.\notag
\end{align}

\end{Hypothesis}


\subsection{Statement of main result}\label{ss:thm}

With all of these hypotheses in hand, we define the desired solution as follows,
\begin{Definition}\label{def:tf}
A \emph{pushed trigger front} is a heteroclinic orbit $U_\mathrm{tf}(\xi;\mu)$ of \eqref{e:sys2} which satisfies the following properties:
\begin{enumerate}
\item
$|U_\mathrm{tf}(\xi;\mu) -  U_p(\xi;\mu)| \rightarrow 0$ and $U_\mathrm{tf}$ converges along the invariant manifold $W^\rcu(U_p)$ as $\xi\rightarrow -\infty$ with asymptotic phase.
\item
 $U_\mathrm{tf}(\xi;\mu)\rightarrow \tl U_*(\mu)$ along the invariant manifold $W^\rs(U_*)$ as $\xi\rightarrow \infty$.
\end{enumerate}
\end{Definition}
Since we only discuss pushed trigger fronts in the rest of this work, we shall henceforth refer to such solutions as just trigger fronts.

\begin{Theorem}\label{t:m2}
Assume Hypotheses \ref{h:sym}--\ref{h:mel} and recall the definition of the eigenvectors $e_{1,\infty}^{\rsu,\rl}$, and $e_{j,\infty}^*$ from Section \ref{ss:spechyp}. Then, there are constants $\rho, L_*>0$ so that for all $L>L_*$ there exists a triggered pushed front $U_\mathrm{tf}(\xi;\mu_*(L))$, with bifurcation curve $\mu_*(L)$, which has the leading order expansion, 
\begin{align}
\mu_*(L) &= -\sum_{j=1,2} \lp[\re^{2\Delta \nu L} d_j + \mathrm{c.c.}\rp] \mc{M}^{-1}e_{j,0}^*+ \mc{O}(\re^{-(2\Delta r+\rho) L})).\notag
\end{align}
Here, 
$$
d_j =  a c_1 \overline{\tl c_j}  \la  e_{1,\infty}^{\rsu,\rl},  e_{j,\infty}^*\ra_\C,\quad \Delta\nu = \nu_\rss - \nu_\rsu,\quad\Delta r = \rre\,\Delta\nu,
$$
the Melnikov mapping $\mc{M}$ is defined in Hypothesis \ref{h:mel}, the constants $a,c_1,\tl c_j\in \C$ are defined in Hypothesis \ref{h:fa}, Lemma \ref{l:ex}, and Lemma \ref{l:exadj} respectively, and $\la\cdot,\cdot\ra_\C$ is the complexified inner product induced by the real inner product on $X$.  Moreover, for each $L$, the elements of the group orbit $\{T_1(\theta)U_\rtf(\xi,\mu_*(L)):\theta\in[0,2\pi)\}$ are also  pushed triggered fronts.

\end{Theorem}
\begin{Remark}\label{r:ddt}
In a typical spatial dynamics formulation, temporal translations form the group orbit of each trigger front, $U_\rtf$.\end{Remark}

\section{Proof of Main Theorem}\label{s:pf}
\subsection{Variational set-up}\label{ss:var}
Our approach to proving Theorem \ref{t:m2} will follow that of Rademacher in \cite{Rademacher05}. There, a gluing-matching procedure akin to Lin's method \cite {lin90} was used to construct solutions near a heteroclinic cycle between a periodic orbit and an equilibrium. Our case is simpler as we glue near a fixed equilibrium, not a periodic orbit.

We wish to construct the desired solution, which connects $U_p$ to $\tl U_*$, by studying variational equations about the heteroclinic orbits corresponding to the preparation front and the pushed front.  For the former, since the heteroclinic $q_2(\xi;\mu)$ is robust in $\mu$, we study variations $w_2(\xi) =  U(\xi)- q_2(\xi;\mu)$ and define the system
\beq\label{e:var2}
\fr{d}{d\xi} w_2 = A_2(\xi) w_2 + g_2(\xi,w_2;\mu),\qquad \xi\in \R,
\eeq
with
$$
A_2(\xi) := D_UF(q_2(\xi;0)), \quad g_2(\xi,w_i) := F(q_2(\xi;\mu) + w_2; \mu) - F(q_2(\xi;\mu);\mu) - A_2(\xi) w_2.
$$ 
For the variations about the pushed front more care must be taken due to the fact that, under our hypotheses, $q_1^0(\xi)$ does not generically persist for all $\mu$ in a neighborhood of the origin.   To deal with this we select a trajectory, $q_1(\xi;\mu)$, defined for $\xi\geq0$, which is contained in the strong-stable manifold $W^\rss(U_*)$, and approaches $q_1^0$ uniformly as $\mu\rightarrow 0$.  This can be done by realizing that trajectories which are near $q_1^0$ and lie in the strong stable manifold are described using the following variation of constants formula
\begin{align}\label{e:sstrj}
v^\rss(\xi;\mu,v_0) &= \Phi^\rss_1(\xi,0)v_0 + \int_0^\xi \Phi_1^\rss(\xi,\zeta)  G_1(\zeta,v^\rss(\zeta);\mu) d\zeta + \int_\infty^\xi \Phi_1^\rsu(\xi,\zeta)G_1(\zeta,v^\rss(\zeta);\mu)d\zeta,\\
G_1(\xi,v;\mu)  &= F(q_1^0(\xi)+ v;\mu) - F(q_1^0(\xi);\mu) - D_UF(q_1^0(\xi);0)v,\notag
\end{align}
where $v_0\in E_1^\rss(0)$. 
In a similar manner we may also define for $\xi\leq 0$, 
\beq\label{e:cutrj}
 v^\rcu(\xi;\mu,v_0) = \Phi^\rcu_{-1}(\xi,0)v_0 + \int_0^\xi \Phi_{-1}^\rcu(\xi,\zeta)  G_1(\zeta,v^\rcu(\zeta);\mu) d\zeta + \int_{-\infty}^\xi \Phi_{-1}^\rs(\xi,\zeta)G_1(\zeta,v^\rcu(\zeta);\mu)d\zeta,
\eeq
where $v_0\in E_{-1}^\rcu(0)$, and $\Phi^{\rcu/\rs}_{-1}$ is the dichotomy associated with the periodic orbit $U_p$ along $q_1^0$ for $\xi\leq 0$.

It then follows for $\mu$ sufficiently small (see \cite[Lem 2.1]{homburg2010homoclinic}) that there exists vectors $v_0^\rss(\mu)\in E_1^\rss(0)$ and $v_0^\rcu(\mu)\in E_{-1}^\rcu(0)$, smooth in $\mu$,  such that $v_0^\rss(0) = v_0^\rsu(0) = 0$ and
\beq\label{e:qadj}
v^\rss(0;\mu,v_0^\rss) - v^\rcu(0;\mu,v_0^\rcu) \in \lp(T_{q_1^0(0)}W^\rss(U_*) + T_{q_1^0(0)}W^\rcu(U_*)\rp)^\perp. 
\eeq
Indeed, this can be obtained by using the Implicit Function theorem to solve the projected equation
\begin{align}
0 &= \mc{Q}\lp[ v^\rss(0;\mu,v_0^\rss) - v^\rcu(0;\mu,v_0^\rcu)  \rp],\notag\\
 &= v_0^\rss - v_0^\rcu +\mc{Q}\lp(\int_\infty^0 \Phi_1^\rsu(\xi,\zeta)G_1(\zeta,v^\rss(\zeta);\mu)d\zeta - \int_{-\infty}^0 \Phi_{-1}^\rs(\xi,\zeta)G_1(\zeta,v^\rcu(\zeta);\mu)d\zeta\rp),\notag
\end{align}
for $v_0^\rss$ and $v_0^\rcu$ in terms of $\mu$, where $\mc{Q}$ is the orthogonal projection of $X$ onto $\lp(T_{q_1^0(0)}W^\rss(U_*) + T_{q_1^0(0)}W^\rcu(U_*)\rp)$.

We shall denote such unique trajectories as 
\begin{align}
q_1(\xi;\mu)&:= v^\rss(\xi;\mu,v_0^\rss(\mu)),\quad \xi\geq0,\notag\\
q_1^-(\xi;\mu)&:= v^\rsu(\xi;\mu,v_0^\rsu(\mu)),\quad \xi\leq 0,\notag
\end{align}
so that $q_1$ approaches $U_*$ along the strong-stable manifold $W^\rss(U_*)$ as $\xi\rightarrow +\infty$  and satisfies $q_1(\xi;0) = q_1^0(\xi)$ for $\xi\geq0$, while $q_1^-$ approaches $U_p$ along the center-unstable manifold $W^\rcu(U_p)$ as $\xi\rightarrow-\infty$ and satisfies $q_1(\xi;0) = q_1^0(\xi)$ for $\xi\leq0$.

%
%
%

We can then define the variation $w_1(\xi) = U(\xi) - q_1(\xi;\mu)$ for $\xi\geq0$ and the variational equation
\beq\label{e:var1}
\fr{d}{d\xi} w_1 = A_1(\xi) w_1 + g_1(\xi,w_1;\mu),\qquad \xi\in \R^+,
\eeq
with
$$
A_1(\xi) := D_UF(q_1(\xi;0);0), \quad g_1(\xi,w_i) := F(q_1(\xi;\mu) + w_1; \mu) - F(q_1(\xi;\mu);\mu) - A_1(\xi) w_1. 
$$ 

%

Next, let $$\Sigma_i = \lp( \fr{dq_i^0}{d\xi}(0)\rp)^\perp,\quad i=1,2$$ be fixed transverse sections to $q_i$, with $\xi$ chosen so that each $\Sigma_i$ lies in a small neighborhood of $U_*$ and the orthogonal complement is taken in $X$. In order to construct the trigger front we wish to find solutions $w_1(\xi)$ and $w_2(\xi)$ of the variational equations in \eqref{e:var2} and \eqref{e:var1} which lie in $W^\rcu(U_p)$ and $W^\rs(\tl U_*)$ respectively and satisfy the following ``gluing" condition for some $L>0$:
\beq\label{e:glue}
w_2(-L) - w_1(L) = q_1(L;\mu) - q_2(-L;\mu).
\eeq

If these conditions hold then the corresponding solutions $U_i$ of \eqref{e:sys2} satisfy
$$
U_1(L) = U_2(-L),\qquad |U_1(-\xi) \rightarrow U_p(\xi)| + |U_2(\xi) - \tl U_*| \rightarrow0,\quad \xi \rightarrow \infty,
$$
so that the solution composed of the concatenation of $U_1$ and $U_2$ is the desired heteroclinic. 
Also, the smoothness of $F$ gives the following pointwise estimates on the variational nonlinearities
\begin{Lemma}\label{h:g}
There exists constants $C_i>0$ such that $g_i$ and its derivative $D_{w_i} g_i$ satisfy the following estimates for all $\xi$ and sufficiently small $w_i\in X$ and $\mu\in \R^2$,
\begin{align}
|g_i(\xi_i,w_i;\mu)| &\leq C \lp(|w_i|^2+ |\mu||w_i|\rp), \\
|D_{w_i} g_i(\xi_i,w_i;\mu)| &\leq C ( |w_i|+|\mu|).
\end{align}
\end{Lemma}
\begin{Proof}
This follows from the assumptions on $F$ and the heteroclinic solutions $q_i$ above.
\end{Proof}

We construct solutions $w_i(\xi)$ to \eqref{e:var1} and \eqref{e:var2} separately, with each satisfying Silnikov boundary conditions for sufficiently large $L>0$:
\begin{align}
P_1^\rss(0)w_1(0) &= \mf{s}_1,\quad P_1^\rsu(L)w_1(L) = \mf{u}_1,\label{e:sil1} \\
P_2^\rss(-L)w_2(-L) &= \mf{s}_2,\quad P_2^\rsu(0)w_2(0) = \mf{u}_2,\label{e:sil2}
\end{align}
where $\mf{u}_i, \mf{s}_i\in X$ are free variables satisfying
\begin{align}
\mf{s}_1\in E^\rss_1(0), &\qquad \mf{u}_1\in E^\rsu_1(L),\label{e:bc1}\\
\mf{s}_2\in E^\rss_2(-L), &\qquad \mf{u}_2\in E^\rsu_2(0).\label{e:bc2}
\end{align}
Also, we require that $w_i(0)\in \Sigma_i$. To simplify notation, let $\mf{W}_i := (\mf{s}_i,\mf{u}_i)$ for $i = 1,2$.

With these solutions we follow the gluing-matching procedure used in \cite{Rademacher05}, which is outlined below and depicted in Figure \ref{f:gl}.
\begin{itemize}
\item \textbf{Section \ref{ss:sil} (Silnikov Solutions)}: Use variation of constants formulas to prove existence of variational solutions $w_i(\mf{W}_i;\mu,L)$, lying near $q_1$ and $q_2$, which lie in certain exponentially weighted function spaces and satisfy the boundary conditions \eqref{e:bc1}-\eqref{e:bc2}.
\item  \textbf{Section \ref{ss:gl} (Gluing)}: Use the gluing condition \eqref{e:glue} to solve for the ``outer" boundary variables $\mf{W}^0 := (\mf{s}_1,\mf{u}_2)$ in terms of the ``inner" boundary variables $\mf{W}^L:=(\mf{s}_2,\mf{u}_1), L$, and $\mu$.
\item  \textbf{Section \ref{ss:tm} (Transverse intersection)}: Match the solution $w_2(\mf{W}^0;\mu,L)$ with $W^\rs(\tl U_*(\mu))$ in the transverse section $\Sigma_2$ of $q_2(0)$.
\item   \textbf{Section \ref{ss:ntm} (Non-transverse intersection)}: Match the the solution $w_1(\mf{W}^0;\mu,L)$ with $W^\rcu(U_p)$ in the transverse section $\Sigma_1$ of $q_1(0)$ by first solving the matching condition in $E_1:=E_1^\rss(0) + E_{-1}^\rcu(0)$ where $E_{-1}^\rcu(0):=T_{q_1(0)}W^\rcu(U_p)$. Then solve the condition in the complement $E_1^\perp$ using Melnikov integrals.
\end{itemize}
In Section \ref{ss:lbf} we then derive asymptotics which allow us to obtain the bifurcation curve discussed in Theorem \ref{t:m2}.

\begin{figure}[h!] \begin{subfigure}[h]{0.9\textwidth}
\hspace{0.3in}
\includegraphics[width=\textwidth]{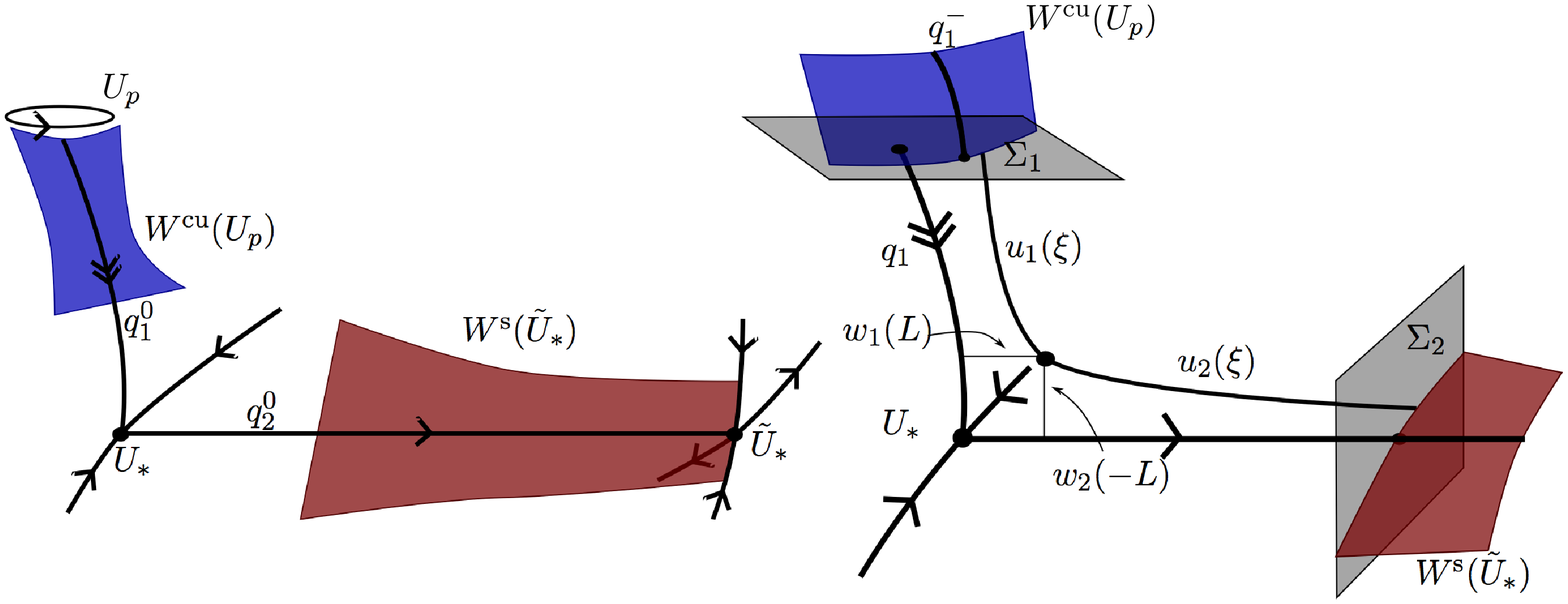}
\end{subfigure}
\hspace{-0.1in}
\begin{subfigure}[h]{0.9\textwidth}
\hspace{0.2in}
\includegraphics[width=\textwidth]{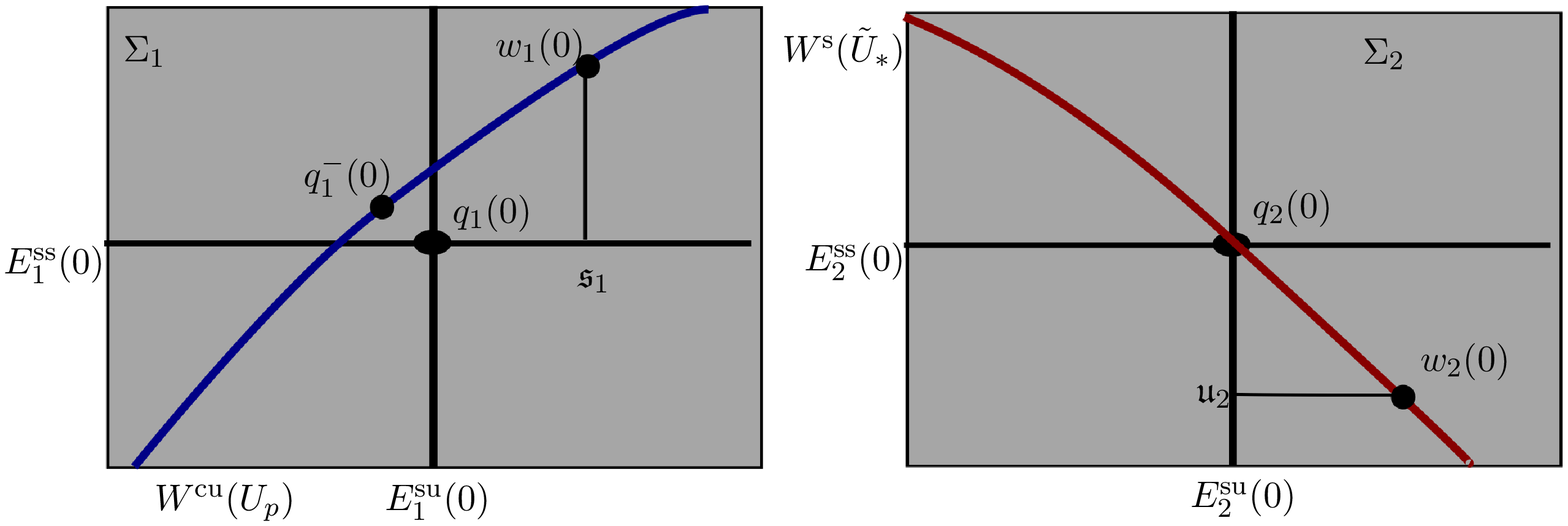}
\end{subfigure}
\hspace{-0.0in}

\caption{Schematic diagram of gluing construction. The top left figure depicts the global phase portrait in $X$ for $\mu = 0$, showing the two manifolds we wish to connect.  The top right figure depicts the gluing construction near the equilibrium $U_*$ for $\mu$ close to $0$, where initial data are taken in transverse sections $\Sigma_{1},\Sigma_{2}$.  These sections are depicted in the bottom figure with the corresponding Silnikov data $\mf{s}_1,\mf{u}_2$ prescribed in each.}
\label{f:gl}

\end{figure}

\subsection{Silnikov Solutions}\label{ss:sil}

In order to find solutions with the desired decay, we use exponentially weighted norms. Let $\eta_\rss, \eta_\rsu>0$ be fixed constants such that $\delta_\rss := r_\rss - \eta_\rss$ and $\delta_\rsu:= \eta_\rsu - r_\rsu  $ are positive and arbitrarily small.  Also define the quantity $m := \eta_\rss - 2\eta_\rsu$, which quantifies the size of the spectral gap $\Delta \eta:= \eta_\rss - \eta_\rsu$, so that $m<0$ if and only if $\eta_\rss/\eta_\rsu <2$.  

%

To begin we define the $L$-dependent norms
\begin{align}
||w_1||_{1,L} &:= \sup_{\xi\in I_1} \re^{\eta_\rss \xi + \gso L}|w_1^\rss(\xi)| +  \sup_{\xi\in I_1} \re^{\eta_\rsu \xi + \guo L}|w_1^\rsu(\xi)|\\
||w_2||_{2,L} &:= \sup_{\xi \in I_2} \re^{\eta_\rss(\xi+L) + \gst L}|w_2^\rss(\xi)|+\sup_{\xi \in I_2} \re^{\eta_\rsu(\xi+L)+ \gut L} |w_2^\rsu(\xi)|
\end{align}
where $I_1 = [0,L]$, $I_2 = [-L,0]$, $w_i^j(\xi) := P_i^j(\xi) w_i(\xi)$ for $i = 1,2$, $j = \rss,\rsu$ and 
$$
\gso =  \Delta\eta + |m| -\eta_\rsu,\quad \guo = \Delta\eta + |m| - \eta_\rsu , \quad \gst = \Delta\eta + |m|, \quad \gut = \Delta \eta + |m| + \rho,
$$
with $\rho>0$ arbitrarily small. Note by the definition of $m$, these quantities are all positive.  
\begin{Remark}\label{r:norm}
These norms were determined in order to make the upcoming fixed point operators uniform contractions in a sufficiently small neighborhood of the origin, and accommodate for all sizes of the spectral gap $\Delta \eta$ in the gluing-matching procedure.  We remark that more general conditions on $\gamma_{\rss/\rsu}$ and $\kappa_{\rss/\rsu}$ can be determined, but we have omitted them for the sake of presentation.

\end{Remark}

We denote 
\begin{align}
\Gamma_{1,L}^\epsilon &:= \{ w_1\in C^k([0,L],X)\, :\, ||w_1||_{1,L}< \epsilon\,\}, \notag\\
\Gamma_{2,L}^\epsilon &:= \{ w_2\in C^k([-L,0],X)\, :\, ||w_2||_{2,L}< \epsilon\notag\},
\end{align}
and define the variation-of-constants operators $$\Psi_i(w_i,\mf{W}_i;\mu, L):=\Psi_i^\rss(w_1,\mf{W}_1;\mu, L) + \Psi_i^\rsu(w_1,\mf{W}_1;\mu, L)$$ with
$$
\left(\begin{array}{c} \Psi_1^\rss(w_1;\mf{W}_1,\mu, L)(\xi)\\ \Psi_1^\rsu(w_1;\mf{W}_1,\mu, L)(\xi)\end{array}\right):= 
\left(\begin{array}{c} \Phi_1^\rss(\xi,0) \mf{s}_1 + \int_0^\xi \Phi_1^\rss(\xi,\zeta) g_1(\zeta,w_1(\zeta);\mu) d\zeta\\
 \Phi_1^\rsu(\xi,L) \mf{u}_1 + \int_L^\xi \Phi_1^\rsu(\xi,\zeta)g_1(\zeta,w_1(\zeta);\mu)d\zeta\end{array}\right),
$$
$$
\left(\begin{array}{c} \Psi_2^\rss(w_2;\mf{W}_2,\mu, L)(\xi)\\ \Psi_2^\rsu(w_2;\mf{W}_2,\mu, L)(\xi)\end{array}\right):= 
\left(\begin{array}{c} \Phi_2^\rss(\xi,-L) \mf{s}_2 + \int_{-L}^\xi \Phi_2^\rss(\xi,\zeta) g_2(\zeta,w_2(\zeta);\mu) d\zeta\\
 \Phi_2^\rsu(\xi,0) \mf{u}_2 + \int_0^\xi \Phi_2^\rsu(\xi,\zeta) g_2(\zeta,w_2(\zeta);\mu)d\zeta\end{array}\right).
$$

Finally, for some small $\delta>0$ we denote 
\begin{align}
\Lambda_{ L}^\delta &= \{\mu\in \R^2 \, : \, \re^{( 2\Delta\eta-\rho) L} |\mu|< \delta\},\quad 
X_{\eta,L}^\delta= \{ u\in X\,:\, \re^{\eta L} |u|<\delta\}.\label{e:XL}
\end{align}
Here, the weight in the parameter space $\Lambda_L^\delta$ was chosen to capture the $\mc{O}(\re^{2\Delta\nu L})$ leading-order dynamics of the bifurcation equation for $\mu$.
It is readily seen that if $w_i$ is a fixed point of $\Psi_i$ then it must also solve \eqref{e:var1} or \eqref{e:var2} for $i = 1,2$. We then have the following proposition,
\begin{Proposition}\label{p:fp}
There exists an $\epsilon_0>0$ such that the following holds.  For all $0<\epsilon<\epsilon_0$, and $L$ sufficiently large, there exists a $\delta>0$ such that for all $\mu \in \Lambda_{L}^\delta$ and $\mf{W}_1 = (\mf{s}_1,\mf{u}_1)$ which satisfy
\beq\label{e:bvp}
\mf{s}_1\in X_{\gso,L}^{\epsilon/6},\qquad \mf{u}_1\in X_{\eta_\rsu+\guo,L}^{\epsilon/6},
\eeq
the variational equation \eqref{e:var1} has a unique solution $w_1^*(\mf{W}_1,\mu,L)\in \Gamma_{1,L}^\epsilon$ which is $C^k$ in $(\xi,\mf{W}_1,\mu)$ and satisfies the boundary conditions \eqref{e:bc1}.

\end{Proposition}

\begin{proof}

We prove this result by showing that the operator, 
$$
\Psi_1(\mf{W}_1): \Gamma_{1,L}^\epsilon \times \Lambda_L^\delta\rightarrow \Gamma_{1,L}^\epsilon,
$$
which can readily be shown to be $C^k$ in both arguments, is well-defined and a uniform contraction.   The resulting unique fixed point will then be the desired solution.

We shall need the following pointwise estimates on the the nonlinearity $g_1$.   For any $\xi \in [0,L]$ with $w\in X$ and $\mu\in \R^2$ sufficiently small, Hypothesis \ref{h:fa} and Lemma \ref{h:g} give a constant $C_1>0$ such that
\begin{align}\label{e:pt2}
|g_1&(\xi,w;\mu)| \leq C_1( |w(\xi)|^2+|w(\xi)||\mu| )  \notag\\
&\leq C_1\Big( |w^\rss(\xi)|^2+ |w^\rsu(\xi)|^2+ |\mu|\left(|w^\rss(\xi)|+ |w^\rsu(\xi)|\right) \Big)\notag\\
&\leq C_1\Big( \re^{-2\eta_\rss \xi- 2\gso L} ||w^\rss||_{1,L}^2+\re^{-2\eta_\rsu \xi - 2\guo L} ||w^\rsu||_{1,L}^2 + |\mu|\left(\re^{-\eta_\rss\xi - \gso L}||w^\rss||_{1,L} + \re^{-\eta_\rsu \xi - \guo L} ||w^\rsu||_{1,L}\right)\Big).
\end{align}
Note that $w^j(\xi) = P_1^j(\xi) w(\xi)$.
Similarly, for $w,v\in X$ we have the quadratic estimate 
\begin{align}\label{e:pt1}
|g_1(\xi,w;\mu) - g_1(\xi,v;\mu)| &\leq C_1'\Big(|w(\xi)| + |v(\xi)| \cdot |w(\xi) - v(\xi)|+ |\mu||w(\xi) - v(\xi)|\Big)\notag\\
&\leq C_1'\re^{-2\eta_\rss\xi - 2\gso L} \lp(||w^\rss||_{1,L} + ||v^\rss||_{1,L} \rp)\cdot ||w^\rss - v^\rss||_{1,L} \notag\\
&\qquad\qquad\qquad\qquad + C_1'\re^{-2\eta_\rsu\xi - 2\guo L}\lp(||w^\rsu||_{1,L} + ||v^\rsu||_{1,L} \rp)\cdot ||w^\rsu - v^\rsu||_{1,L} \notag\\
&\qquad\qquad\qquad\qquad+ C_1'|\mu|\left(\re^{-\eta_\rss\xi - \gso L}||w^\rss - v^\rss||_{1,L} + \re^{-\eta_\rsu\xi - \guo L} ||w^\rsu - v^\rsu||_{1,L}\right).
\end{align}

We then find
\begin{align}\label{e:fp1}
\sup_{\xi\in I_1}&\,\,\re^{\eta_\rss \xi + \gso L}|\Psi_1^\rss(w_1)(\xi)| \leq 
\sup_{\xi\in I_1}\,\, \re^{(\eta_\rss - r_\rss)\xi + \gso L} |\mf{s}_1|+
\re^{\eta_\rss \xi + \gso L} \int_0^\xi \re^{-r_\rss(\xi - \zeta)} C_1\lp( |w_1(\zeta)|^2+|w_1(\zeta)||\mu|  \rp)d\zeta.
\end{align}
The condition on $\mf{s}_1$ gives
\beq\label{e:fp1a}
\sup_{\xi\in I_1}\,\, \re^{(\eta_\rss - r_\rss)\xi + \gso L} |\mf{s}_1| \leq \fr{\epsilon}{6}.
\eeq
Next we estimate the term involving $|w_1|^2$ in \eqref{e:fp1}:
\begin{align}\label{e:fp2}
\sup_{\xi\in I_1}&\,\,\re^{\eta_\rss \xi + \gso L} \int_0^\xi \re^{-r_\rss(\xi - \zeta)} C_1 |w_1(\zeta)|^2 d\zeta\notag\\
&\leq  C_1 \lp( \re^{-(\gso + \delta_\rss)L} - \re^{-(\gso +\eta_\rss)L} \rp)\fr{||w^\rss||_{1,L}^2}{2\eta_{\rss} - r_\rss}+ C_1 \lp( \re^{(\gso - 2 \guo + \eta_\rss - 2 \eta_\rsu)L} - \re^{(\gso - 2 \guo - \delta_\rss)L} \rp)\fr{||w^\rsu||_{1,L}^2}{|2\eta_\rsu - r_\rss|} \notag\\
  &\leq 2C_1 C_1^*||w_1||_{1,L}^2 <\epsilon/6 .
\end{align}
where $C_* = \max\{ \fr{1}{2\eta_\rss - r_\rss}, \fr{1}{|2\nu_\rsu - r_\rss|}\}$, and we require $\epsilon<\fr{1}{12C_1C_*}$.  Also note that we have used the fact that $\gso - 2 \guo = 2\eta_\rsu-\eta_\rss - |m| -2\rho$.

The term with $|w_1||\mu|$ in \eqref{e:fp1} can similarly be estimated by
\begin{align}\label{e:fp3}
C_1\sup_{\xi\in I_1}\,\,&\re^{\eta_\rss \xi + \gso L} \int_0^L \re^{-r_\rss(\xi - \zeta)}  |\mu| \left(\re^{-\eta_\rss\zeta - \gso L}||w^\rss ||_{1,L} + \re^{-\eta_\rsu\zeta - \guo L} ||w^\rsu||_{1,L}\right)\notag\\
&\leq C_1|\mu|\lp(  \re^{-\delta_\rss L}||w^\rss||_{1,L} + \re^{(\gso - \guo + \Delta\eta)L}||w^\rsu||_{1,L} \rp)\notag\\
&\leq C_1\delta||w||_{1,L}< \epsilon/6,
\end{align}
for any $0<\delta < \fr{1}{6C_1}$, since $\mu\in \Lambda_L^\delta$ and $\gso - \guo =-2\rho$. 

Combining \eqref{e:fp1a}, \eqref{e:fp2}, and \eqref{e:fp3} we obtain 
\beq
||\Psi_1^\rss(w_1)||_{1,L} <\epsilon/2.
\eeq

Similar estimates may be applied to obtain
\beq
||\Psi_1^\rsu(w_1)||_{1,L} < \epsilon/2,
\eeq
for any $\epsilon< (C_1 \max\{ \fr{1}{2\eta_\rss - r_\rsu}, \fr{1}{2\eta_\rsu - r_\rsu} \})^{-1}$. These can then be combined to obtain $||\Psi_1(w)||_{1,L}<\epsilon$.

To prove the contraction, the pointwise estimate \eqref{e:pt1} can be used in a similar way to obtain 
\begin{align}
||\Psi_1(w;\mf{W}_1,\mu,L) - \Psi_1(v;\mf{W}_1,\mu,L)||_{1,0} \leq C_1 C_* \lp(4\epsilon + |\mu|\re^{\Delta\eta L}\rp)||w - v||_{1,L}< \fr{1}{2}||w - v||_{1,L},\qquad w,v\in \Gamma_{1,L}^\epsilon,
\end{align}
For $\epsilon_0$ sufficiently small, and $L$ sufficiently large. Since $\Psi_1$ is smooth in $\mu$ and $\mf{W}_1$, the Uniform Contraction principle then gives the result.
\end{proof}

An analogous proof gives the existence of a solution for \eqref{e:var2}.
\begin{Proposition}\label{p:fp2}
There exists an $\epsilon_0>0$ such that the following holds:  For all $0<\epsilon<\epsilon_0$ and $L$ sufficiently large there exists a $\delta>0$ such that for all $\mu \in \Lambda_{L}^\delta$ and $\mf{W}_2 = (\mf{s}_2,\mf{u}_2)$ which satisfy
\beq\label{e:bvp2}
\mf{s}_2\in X^{\epsilon/6}_{\gst,L},\qquad \mf{u}_2\in X_{\eta_\rsu+\gut,L}^{\epsilon/6},
\eeq
equation \eqref{e:var2} has a unique solution $w_2^*(\mf{W}_2,\mu,L)\in \Gamma_{2,L}^\epsilon$ which is $C^k$ in $(\xi,\mf{W}_2,\mu)$.

\end{Proposition}

\subsection{Gluing}\label{ss:gl}
We now wish to find boundary data $\mf{W}_1 = (\mf{s}_1,\mf{u}_1)$ and $\mf{W}_2 = (\mf{s}_2,\mf{u}_2)$ for which the solutions $w_1^*,w_2^*$ satisfy the gluing equation \eqref{e:glue}.  We use the projections $P_2^{\rss}(-L)$ and $P_1^{\rsu}(L)$ to decompose the gluing equation \eqref{e:glue} into the system
\begin{align}\label{e:gl}
P_2^{\rss}(-L)w_{1}^{*}(L) - \mf{s}_2 &= P_2^\rss(-L)\Delta q(L)\notag\\
\mf{u}_1 - P_1^\rsu(L)w_2^*(-L) &= P_1^\rsu(L)\Delta q(L)
\end{align}
where $\Delta q(L) = q_2(-L)- q_1(L)$.  

To simplify this system, we use the following estimates on the $\xi$-dependent projections.
\begin{Lemma}\label{l:proj}
For $L$ sufficiently large there exists a constant $K>0$ such that
\begin{align}
|P_1^j(L) u - P_{1,\infty}^ju|\leq K \re^{-\Delta \eta L}|u|,\\
|P_2^j(-L)u - P_{1,\infty}^j u|\leq K \re^{-\Delta\eta L}|u|,
\end{align}
for $j = \rss,\rsu$, $u\in X$.
\end{Lemma}
\begin{proof}
Using the asymptotic decay of $A_i(\xi)$ as $\xi\rightarrow \pm\infty$ for $i = 1,2$ respectively, this result follows from \cite[Cor. 2]{Peterhof97}.
\end{proof}

From this lemma, the heteroclinic asymptotics in Hypothesis \ref{h:fa} then give that \eqref{e:gl} has the leading order form
\begin{align}
\mf{s}_2 &= \lp[w_1^{*,\rss}(L)-  q_1(L)\rp] \lp(1+\mc{O}(\re^{-\Delta\eta L})\rp) \notag\\
\mf{u}_1 &= \lp[w_2^{*,\rsu}(-L)+ q_2(-L)\rp]\lp(1+\mc{O}(\re^{-\Delta\eta L})\rp). \notag
\end{align}

Hence, it suffices to prove the existence of solutions to the truncated system
\begin{align}
\mf{s}_2 &=w_1^{*,\rss}(L)-  q_1(L) \notag\\
\mf{u}_1 &=w_2^{*,\rsu}(-L)+ q_2(-L). \notag
\end{align}
Such solutions can be found as fixed points of the following operator,
\begin{align}\label{e:fpgl}
H_\mathrm{gl}&:\mf{X}_L^\epsilon\times \mf{X}_0^\epsilon \times \Lambda_L^\delta \rightarrow \mf{X}_L^\epsilon\notag\\
& (\mf{W}^L;\mf{W}^0,\mu)\mapsto 
\left(\begin{array}{c} \Psi_1^\rss(w_1^*(\mf{s}_1,\mf{u}_1, \mu,L),\mu,L)(L) \\
\Psi_2^\rsu(w_2^*(\mf{s}_2,\mf{u}_2, \mu,L);\mu, L)(-L)\end{array}\right) +
 \left(\begin{array}{c}-q_1(L;\mu)\\
q_2(L;\mu) \end{array}\right),
\end{align}
where we solve for the inner boundary values near the equilibrium $U_*$,
$$
\mf{W}^L:=(\mf{s}_2,\mf{u}_1)\in\mf{X}_L^\epsilon:=X_{\gst,L}^\epsilon\times X_{\eta_\rsu+\guo,L}^\epsilon,
$$
in terms of the outer boundary values near $q_1(0)$ and $q_2(0)$,
$$
\mf{W}^0:=(\mf{s}_1,\mf{u}_2)\in\mf{X}_0^\epsilon:=  X_{\gso ,L}^\epsilon\times X_{\eta_\rsu+\gut, L}^\epsilon.
$$
The exponential weights on these values are chosen to be consistent with the contraction arguments in Propositions \ref{p:fp} and \ref{p:fp2}. We then obtain the following existence result,

\begin{Proposition}\label{p:gl}
There exists an $\epsilon_1>0$ such that the following holds.
For all $\epsilon< \epsilon_1$, $L$ sufficiently large, there exists a $\delta>0$ such that the mapping $H_\mathrm{gl}:\mf{X}_L^\epsilon\times \mf{X}_0^\epsilon \times \Lambda_L^\delta \rightarrow \mf{X}_L^\epsilon$ has a unique fixed point $ \mf{W}^L_\dagger(\mf{W}^0,\mu,L) := \lp(\mf{s}_2^\dagger(\mf{W}^0,\mu,L),\mf{u}_1^\dagger(\mf{W}^0,\mu,L)\rp)$ which is $C^k$ smooth in all its variables.
\end{Proposition}

\begin{proof}
First note that since $w_i^*$ is smooth in $\mf{W}_i$ and $\mu$, and $w_i^*(0;0,L)= 0$, we have that 
\begin{align}\label{e:gl3}
||w_1^*(\mf{W}_1;\mu,L)||_{1,L}&\leq C_3 \lp( \re^{\gso L}|\mf{s}_1|+ \re^{(\eta_\rsu+\guo)L}|\mf{u}_1| +\re^{\Delta \eta L}|\mu|\rp),\notag\\
||w_2^*(\mf{W}_2;\mu,L)||_{2,L}&\leq C_3 \lp( \re^{\gst L}|\mf{s}_2|+ \re^{(\eta_\rsu+\gut)L}|\mf{u}_2| +\re^{\Delta \eta L}|\mu|\rp).
\end{align}

Similar to the proofs in the previous section, the pointwise estimates \eqref{e:pt2} and \eqref{e:pt1} then give the estimates
\begin{align}
\re^{\gst L} |\Phi_1^\rss(L,0) \mf{s}_1| &\leq K_1 \re^{(\gst - r_\rss)L}|\mf{s}_1|,\notag\\
\re^{\gst L} |\int_0^L \Phi_1^\rss(L,\zeta)& g_1(\zeta,w_1^*;\mu) d\zeta|\leq
 C_2 \re^{\gst L}\lp(  \re^{-2(\eta_\rss + \gso)L} ||w_{1}^\rss||^2_{1,L} + \re^{-2(\eta_\rsu + \guo)L} ||w_1^\rsu||_{1,L}^2  \rp)\notag\\
 &\qquad\qquad\qquad\qquad\qquad\qquad\qquad+ |\mu|\lp( \re^{-(\eta_\rss+\gso) L}||w_1^\rss||_{1,L}  + \re^{-(\eta_\rsu + \guo)L} ||w_1^\rsu||_{1,L}\rp), \notag
\end{align}
for the first component of $H_\mathrm{gl}$. 

For the second component we obtain for some constants $K_2, K_2'>0$,
\begin{align}
\re^{(\eta_\rsu+\guo) L} |\Phi_2^\rsu(-L, 0) \mf{u}_2| &\leq K_2 \re^{(2\eta_\rsu + \guo)L} |\mf{u}_2|\leq K_2\re^{(\eta_\rsu + \guo - \gut)L} \epsilon , \label{e:est11}\\
\re^{(\eta_\rsu+\guo) L} |\int_0^{-L} \Phi_2^\rsu(-L,\zeta) g_2(\zeta,w_2^*;\mu) d\zeta| &\leq
K_2\re^{(2\eta_\rsu+\guo - \delta_\rsu) L} \Bigg(  \re^{-2(\eta_\rss + \gst)L}||w_2^\rss||_{2,L}^2 + \re^{-2(\eta_\rsu+\gut)L} ||w_2^\rsu||_{2,L}^2   \notag\\ 
&\quad+ |\mu| \lp (  \re^{-(\eta_\rss + \gst)L}||w_2^\rss||_{2,L} + \re^{-(\eta_\rsu + \gut)L}||w_2^\rsu||_{2,L}  \rp) \Bigg)\notag\\
&\leq K_2\Bigg(  \re^{ (-3\Delta\eta - |m| - \eta_\rsu )L}||w_2^\rss||_{2,L}^2 + \re^{(-\eta_\rsu - \Delta \eta - |m|)L} ||w_2^\rsu||_{2,L}^2   \notag\\ 
&\quad+ |\mu| \lp (  \re^{ -\Delta\eta L}||w_2^\rss||_{2,L} + \re^{-(\delta_\rsu + \rho)L}||w_2^\rsu||_{2,L}  \rp) \Bigg),\notag\\
\re^{\eta_\rsu + \guo L} | q_2(L)| &\leq K_2'\re^{(\Delta\eta + |m| - r_\ru) L}. \notag
\end{align}
Pairing these estimates with those in \eqref{e:gl3} and using Hypothesis \ref{h:eig1}, we can then obtain the desired invariance of $H_{\mathrm{gl}}$ for sufficiently small $\epsilon$.

Uniform Contraction then follows in a similar manner, using the fact that $||w_i^*(\mf{W}_i) - w_i^*(\mf{V}_i)||_{i,L} \leq C_2 |\mf{W}_i - \mf{V}_i|$ for small enough $\mf{W}_i,\mf{V}_i$. 
\end{proof}

This proposition implies the existence of boundary data
$$
 \mf{W}^L_\mathrm{gl}(\mf{W}^0,\mu,L) := \lp(\mf{s}_2^\mathrm{gl}(\mf{W}^0,\mu,L),\mf{u}_1^\mathrm{gl}(\mf{W}^0,\mu,L)\rp),
$$ smooth in all dependent variables, which give glued solutions  
$$
w_1^\mathrm{gl}(\mf{W}^0,\mu,L)(\xi) := w_1^*(\mf{s}_1,\mf{u}_1^\mathrm{gl}(\mf{W}^0), \mu,L)(\xi),\quad 
w_2^\mathrm{gl}(\mf{W}^0,\mu,L)(\xi) := w_2^*(\mf{s}_2^\mathrm{gl}(\mf{W}^0),\mf{u}_2, \mu,L)(\xi),
$$
satisfying \eqref{e:gl}.

\subsection{Transverse matching}\label{ss:tm}

Now we match the glued solution with the stable manifold $W^\rs(\tl U_*)$ inside $\Sigma_2$. Since this manifold is $C^k$-smooth and intersects $W^\rsu(U_*)$ transversely, we have that $W^\rs(\tl U_*)\cap \Sigma_2$ can be locally described near $q_2(0)$ as a graph $h_2: E_{+2}^\rs(0)\cap\Sigma_2\rightarrow E_2^\rsu(0)\cap \Sigma_2,$ where $E_{+2}^\rs(0):=P_{+2}^\rs(0)X$ and 
$$
|h_2(v_2;\mu)|\leq K_\rs |v_2|(|\mu|+|v_2|),
$$ 
for some $K_\rs>0$, and sufficiently small $\mu$ and $v_2\in E_{+2}^\rs(0)\cap\Sigma_2$.   Here, $P_{+2}^\rs(\xi)$ is the projection associated with the dichotomy used to construct the invariant manifold $W^\rs(\tl U_*)$. 
We thus obtain the matching equation
\beq\label{e:tr}
w_2^\mathrm{gl}(\mf{W}^0;\mu,L)(0) = v_2+ h_\rs(v_2;\mu), 
\eeq
which we use to solve for $(\mf{u}_2, v_2)$ in terms of $(\mf{s}_1,\mu)$.
Defining 
$$
\tl g_i(\xi,\mf{W}^0;\mu) := g_i(\xi, w_i^\mathrm{gl}(\mf{W}^0;\mu,L)(\xi);\mu),\quad i = 1,2,
$$ we use the projected solution operators $\Psi_2^{\rss/\rsu}$ to write \eqref{e:tr} as
\beq\label{e:tr1}
\mf{u}_2 - v_2   = 
h_\rs( v_2 ,\mu) -\Phi_2^\rss(0,-L)\mf{s}_2^\mathrm{gl}(\mf{W}^0) -\int_{-L}^0 \Phi_2^\rss(0,\zeta)\tl g_2(\zeta;\mf{W}^0;\mu) d\zeta.
\eeq

\begin{Proposition}\label{p:tr}
For some $\epsilon_2>0$ the following holds:
For all $0<\epsilon<\epsilon_2$ there exists a $\delta>0$ such that \eqref{e:tr} has a $C^k$-solution $(\mf{u}_2^\mathrm{tr},v_2^\mathrm{tr})(\mf{s}_1,\mu)\in X_{\eta_\rsu+\gut,L}^\epsilon\times X_{\eta_\rsu+\gut,L}^\epsilon$, for each $(\mf{s}_1,\mu)\in X_{\gso,L}^{\epsilon/8}\times \Lambda_{L}^{\delta}$. 
\end{Proposition}
\begin{proof}
First, Hypothesis \ref{h:tan} implies that the canonical mapping
\begin{align}
S_2:& (E_2^\rsu(0)\cap \Sigma_2) \times  (E_{+2}^\rs(0)\cap\Sigma_2) \rightarrow \Sigma_2\notag\\
&(\mf{u}_2, v_2)\mapsto \mf{u}_2 - v_2 ,\notag
\end{align} is invertible with uniformly bounded inverse.  Thus \eqref{e:tr} can be rewritten as the fixed point problem
\begin{align}
(\mf{u}_2,v_2) &= S_2^{-1}\Big(h_\rs(v_2,\mu) -\Phi_2^\rss(0,-L)\mf{s}_2^\mathrm{gl}(\mf{W}^0) -\int_{-L}^0 \Phi_2^\rss(0,\zeta)\tl g_2(\zeta,\mf{W}^0;\mu) d\zeta\Big)\notag\\
&= \lp(  h_\rs(v_2;\mu)\,,\, -P_{+2}^\rs(0)\lp( \Phi_2^\rss(0,-L)\mf{s}_2^\mathrm{gl}(\mf{W}^0)+\int_{-L}^0 \Phi_2^\rss(0,\zeta)\tl g_2(\zeta,\mf{W}^0;\mu) d\zeta  \rp)  \rp)\label{e:trfp1}.
\end{align}
since $S^{-1}_2(x) = (P_{2,+}^\rsu(0) x, -P_{+2}^\rs(0)x)$, where $P_{2,+}^\rsu(\xi)$ was defined in Hypothesis \ref{h:inc} above.

We then obtain the following estimates on the different terms of the right-hand side of the above equation for some constant $K>0$,
\begin{align}
\re^{(\eta_\rsu+\gut) L} \Big|h_\rs(v_2 ,\mu) \Big| &\leq K\re^{(\eta_\rsu+\gut) L} \lp(|v_2||\mu| + |v_2|^2 \rp) \notag\\
&\leq K\epsilon(\delta + \epsilon),\label{e:est22}\\
\re^{(\eta_\rsu+\gut)L} |P_{+2}^\rs(0) w_2^\rss(0)|&\leq K \re^{(\eta_\rsu + \gut - \eta_\rss - \gst)L} ||w_2^\rss||_{2,L} =  K\re^{(\rho-\Delta \eta) L}   ||w_2^\rss||_{2,L}.\label{e:trfp}
\end{align}
Applying the estimates in \eqref{e:gl3} we conclude that \eqref{e:trfp1} is a uniform contraction and thus possesses a unique fixed point.

%
\end{proof}

We denote the subsequent glued solutions which also solve the transverse matching problem as 
$$
w_1^\mathrm{tr}(\mf{s}_1,\mu,L)(\xi):= w_1^\mathrm{gl}(\mf{s}_1, \mf{u}_2^\mathrm{tr}, \mu,L)(\xi),\quad
w_2^\mathrm{tr}(\mf{s}_1,\mu,L)(\xi):= w_2^\mathrm{gl}(\mf{s}_1, \mf{u}_2^\mathrm{tr},\mu,L)(\xi),
$$
where $\mf{u}_2^\mathrm{tr} = \mf{u_2}^\mathrm{tr}(\mf{s}_1,\mu)$.

\subsection{Non-Transverse matching}\label{ss:ntm}
Now let us match the glued solution with $W^\rcu(U_p)$ in $\Sigma_1$. In a neighborhood of $q_1(0;\mu)$, the intersection of the center-unstable manifold $W^\rcu(U_p)\cap \Sigma_1$ can be described as a graph $q_1^-(0;\mu) + v + h_p(v;\mu)$ with
$$
h_p : E_{-1}^\rcu(0)\cap \Sigma_2 \rightarrow  \tl E_1^\rss(0)\oplus E_1^\perp, \quad v\in E_{-1}^\rcu(0)\cap \Sigma_2
$$
where $E_{-1}^\rcu(0):=P_{-1}^\rcu(0)X$, $E_1:=E_{-1}^\rcu(0)+E_{1}^\rss(0)$, and $\tl E_1^\rss(0) $ is the orthogonal complement of $Z= E_{-1}^\rcu(0)\cap E_1^\rss(0)$ in $E_1^\rss(0)$. Also define the orthogonal projection $Q_1: E_{-1}^\rcu(0)\rightarrow Z\cap\Sigma_2$ and let $\tl v_1 = (I-Q_1) v,\,\, \mf{v}_1 = Q_1 v$.  We remark that since $\fr{d}{d\xi} q_1^0(0) \in Z$, the range of $Q$ is a one-dimensional subspace.

We thus wish to solve the matching equation
\beq\label{e:ntr}
w_1^\mathrm{tr}(\mf{s}_1,\mu,L)(0) = q_1^-(0;\mu) - q_1(0;\mu)+\tl v_1+ \mf{v}_1 + h_p(\tl v_1+ \mf{v}_1;\mu).
\eeq
In order to do this we shall first solve the projected equation on $E_1$ after which we can the solve on the complement $E_1^\perp$ using Melnikov integrals.  

To achieve this first step, we apply the orthogonal projection $P_1: X\rightarrow E_1$ to \eqref{e:ntr} and use \eqref{e:qadj} to obtain
\beq\label{e:ntrp}
\mf{s}_1 - \tl v_1 = \mf{v}_1+ P_1\lp( h_p(\tl v_1+ \mf{v}_1;\mu) - \Phi_1^\rsu(0,L) \mf{u}_1^\mathrm{gl}(\mf{s}_1,\mf{u}_2^\mathrm{tr})- \int_L^0 \Phi_1^\rsu(0,\zeta)\tl g_1(\zeta, \mf{W}^0_\mathrm{tr};\mu)d\zeta    \rp).
\eeq
We then obtain the following result,
\begin{Proposition}\label{p:ntrp}
There exists $\epsilon_3>0$ such that the following holds.
For all $\epsilon<\epsilon_3$ there exists a $\delta>0$ such that there is a unique $C^k$-solution $(\mf{s}_1^\mathrm{m},\tl v_1^\mathrm{m})(\mf{v}_1,\mu,L) \in X_{\gso, L}^\epsilon\times X_{\gso, L}^\epsilon$ of \eqref{e:ntrp} for each $\mf{v}_1\in X^\epsilon_{\gso, L},$ and $\mu\in \Lambda_{L}^\delta$.
 \end{Proposition}
\begin{proof}
The proof follows in a similar manner as in Proposition \ref{p:tr} and we omit it.
\end{proof}

We denote the corresponding solutions as $$w_i^\mathrm{m}(\mf{v}_1,\mu,L)(\xi) = w_i^\mathrm{tr}(\mf{s}_1^\mathrm{m}(\mf{v}_1,\mu,L),\mu,L)(\xi).$$

Now we wish to solve the component of \eqref{e:ntr} in $E_1^\perp$ and thus complete the gluing matching procedure. This can be done by solving the equations
\beq
\la w_1^\mathrm{n}(\mf{v}_1,L)(0), e_{j,0}^* \ra =\la  q_1^-(0;\mu) - q_1(0;\mu) , e_{j,0}^* \ra+  \langle h_p( \mf{v}_1+\tl v^\mathrm{n}_1;\mu) , e_{j,0}^* \rangle, \qquad j = 1,2,\notag
\eeq
where and  $e_{1,0}^*, e_{2,0}^*\in X$ have unit norm and form a basis of $E_1^\perp$.  Also, these basis elements correspond to bounded solutions $e_j^*(\xi)$ of the adjoint variational equation \eqref{e:adjvar} which satisfy $e_j^*(0) = e_{j,0}^*$.   Applying the variation of constants formula to $w_1^\mathrm{n}$ and noticing that $\la P_{1}^\rss w_1^\mathrm{n}(0), e_{j,0}^*\ra = 0$ we then obtain, for $j = 1,2$,
\beq\label{e:mel0}
\lp\langle \Phi_1^\rsu(0,L) \mf{u}_1^\mathrm{tr}, e_{j,0}^* \rp\rangle  =\la  q_1^-(0;\mu) - q_1(0;\mu) , e_{j,0}^* \ra+  \langle h_p( \mf{v}_1+\tl v^\mathrm{n}_1;\mu) , e_{j,0}^* \rangle - \lp\langle \int_L^0 \Phi_1^\rsu(0,\zeta)\tl g_1(\zeta,\mf{W}^0_\mathrm{tr};\mu)d\zeta , e_{j,0}^*\rp\rangle.
\eeq

The expressions from \eqref{e:sstrj} and \eqref{e:cutrj} give that
\begin{align}
\la  q_1^-(0;\mu) - q_1(0;\mu) , e_{j,0}^* \ra &=  \la \int_{-\infty}^0 \Phi_{-1}^\rs(\xi,\zeta)G_1(\zeta,v^\rcu(\zeta);\mu)d\zeta + \int_0^\infty \Phi_1^\rsu(\xi,\zeta)G_1(\zeta,v^\rss(\zeta);\mu)d\zeta, e_{j,0}^* \ra\notag\\
&=\lp( \int_{-\infty}^\infty\la\, D_\mu F(q_1^0(\zeta);0), e^*_j(\zeta) \,\ra d\zeta\rp) \mu + \mc{O}(|\mu|^2).\notag
\end{align}

Now, since there exists constants $ C>0$ such that  
\begin{align}
|h_p(\mf{v}_1+\tl v^\mathrm{n}_1(\mf{v}_1;\mu);\mu)|  &\leq C(|\mf{v}_1|+|\mu|)^2,\notag\\
| \lp\langle \int_L^0 \Phi_1^\rsu(0,\zeta)\tl g_1(\zeta,\mf{W}^0_\mathrm{tr};\mu)d\zeta , e_{j,0}^*\rp\rangle|& \leq
\re^{-\gso L} C\lp( |\mf{v}_1|+ |\mu|\rp)^2
 \end{align}
 for $\mu$ sufficiently small and $L$ sufficiently large, we can use the quadratic estimates on $\tl g_1$ to obtain that \eqref{e:mel0} has the form
\beq\label{e:mel}
\la\Phi_1^\rsu(0,L) \mf{u}_1^\mathrm{tr}, e_{j,0}^* \ra =\int_{-\infty}^\infty\la\, D_\mu F(q_1^0(\zeta);0), e^*_j(\zeta) \,\ra d\zeta \,\,\mu + \mc{O}\lp( (|\mu|+|\mf{v}_1|)^2\rp), 
\eeq
or after rearranging,
\beq\label{e:mel4}
\mc{M}\mu(\mf{v}_1,L) =  \Big(  \sum_{j =1,2}    \la \Phi_{1}^\rsu(0,L) \mf{u}_1^\mathrm{tr}, e_{j,0}^*\ra e_{j,0}^*   \Big)+\mc{O}\lp( (|\mu|+|\mf{v}_1|)^2\rp).
\eeq

Since the Melnikov mapping is invertible by assumption, the implicit function theorem then gives, for $L$ sufficiently large, that there exists a family of solutions $\mu_\mathrm{m}(\mf{v}_1,L)$ with the following leading order expansion
\beq\label{e:mel5}
\mu_\mathrm{\rtf}(\mf{v}_1,L) = \mc{M}^{-1} \Big(  \sum_{j =1,2}    \la \Phi_{1}^\rsu(0,L) \mf{u}_1^\mathrm{tr}, e_{j,0}^*\ra e_{j,0}^*   \Big)+\mc{O}\lp(|\mf{v}_1|   \rp).
\eeq

Let us denote the corresponding glued solution for these parameters as $w_i^\mathrm{\rtf}(\xi) := W_i^\mathrm{\rtf}(\mf{v}_1,\mu_\mathrm{\rtf}(\mf{v}_1,L),L)(\xi)$ and the projections as $w_{i}^{\mathrm{\rtf},\rss/\rsu}(\xi) = P_i^{\rss/\rsu}(\xi) w_i^\mathrm{\rtf}(\xi)$.  From these we immediately obtain the desired heteroclinic solution $U_{\rtf}$ as a concatenation of the solutions $U_{i}^\rtf:= q_i + w_i^\mathrm{\rtf}$.  This gives the existence of a one-parameter family of solutions for each $L$ as described in the statement of the theorem. By uniqueness, the parameter $\mf{v}_1$ must parameterize the group orbit of a solution under the $S^1$-action.


\subsection{Leading order bifurcation equation expansions}\label{ss:lbf} \label{ss:exp}
In this section, we complete the proof of the theorem by obtaining the desired expansion of the bifurcation equation \eqref{e:mel5} obtained in the previous section.  Let us ease notation by denoting $w_i = w_i^\mathrm{\rtf}, w_{i}^{\rss/\rsu} = w_{i}^{\mathrm{\rtf},\rss/\rsu},$ and  $\mu = \mu_\mathrm{\rtf}.$ 
 To obtain finer expansions we isolate the leading components of the $\xi$-dependent subspaces as described in \eqref{e:spl} above,
$$
E_{i}^\rss(\xi) = E_{i}^{\rss,\mathrm{l}}(\xi) +  E_{i}^{\rss,\rs}(\xi),\qquad E_{i}^\rsu(\xi) = E_{i}^{\rsu,\mathrm{l}}(\xi) +  E_{i}^{\rsu,\ru}(\xi), \quad i = 1,2.
$$
In the following we will study these real subspaces using the complexified flow. Here the complexification of the eigenspaces $E_{1,\infty}^{\rss/\rsu,\rl}$ are spanned by the eigenvectors $e_{1,\infty}^{\rss/\rsu,\rl},\,\,\overline{e_{1,\infty}^{\rss/\rsu,\rl}}$ of $D_UF(U_*)$ corresponding to the leading complex-conjugate eigenvalues $\nu_{\rss/\rsu},\,\,\overline{\nu_{\rss/\rsu}}$.  
Before continuing, we prove the following three lemmata which are needed in our derivation of the leading order bifurcation equation.

\begin{Lemma}\label{l:ex}
For $L>0$ sufficiently large there exist constants $c_1,c_2\in \C$ and $\rho>0$ such that the following asymptotic expansions hold
\begin{align}
\Phi_2^\rsu(-L,0)P_{2,+}^{\rsu}(0)\Phi_2^\rss(0,-L) e_{1,\infty}^{\rss,\rl} = 
c_1e^{\Delta \nu L}e_{1,\infty}^{\rsu,\rl} + \mc{O}(\re^{-(\Delta r +\rho) L}),\notag\\
\Phi_2^\rsu(-L,0)P_{2,+}^{\rsu}(0)\Phi_2^\rss(0,-L) \overline{e_{1,\infty}^{\rss,\rl} }= 
c_2\re^{\overline{\Delta\nu}L}\overline{e_{1,\infty}^{\rsu,\rl}}  + \mc{O}(\re^{-(\Delta r +\rho) L}).\notag
\end{align}

\end{Lemma}

\begin{proof}
 
 
First note that by its construction, the restricted dichotomy $\Phi^{\rss,\rl}_2(\xi,\zeta)$ is well-defined for both $\xi\geq \zeta$ and $\xi\leq \zeta$.  Hence there exists vectors $v_1^\rss,v_2^\rss$ which span $E^{\rss,\rl}_2(0)$ such that 
\begin{align}
 \Phi_2^{\rss,\rl}(-L,0)v_1^\rss&= \re^{-\nu_\rss L} e_{1,\infty}^{\rss,\rl} + \mc{O}(\re^{(r_\rss - \rho)L)}),\notag\\
  \Phi_2^{\rss,\rl}(-L,0)v_2^\rss&= \re^{-\overline{\nu_\rss} L} \overline{e_{1,\infty}^{\rss,\rl}} + \mc{O}(\re^{(r_\rss - \rho)L)}).\notag
 \end{align}
 Applying $\Phi_2^{\rss}(0,-L)$ to both sides of these equations, we then obtain
\begin{align}
 \Phi_2^{\rss}(0,-L)  e_{1,\infty}^{\rss,\rl}  =  \re^{\nu_\rss L} v_1^\rss+ \mc{O}(\re^{-(r_\rss + \rho)L)}),\notag\\
 \Phi_2^{\rss}(0,-L)  \overline{e_{1,\infty}^{\rss,\rl}}  =  \re^{\overline{\nu_\rss} L} v_2^\rss+ \mc{O}(\re^{-(r_\rss + \rho)L)}).\notag
\end{align}
 
 In a similar manner, there exists vectors $v_{1}^\rsu,v_2^\rsu$ which span $E_2^{\rsu,\rl}(0)$ such that 
\begin{align}
 \Phi_2^\rsu(-L,0) v_1^\rsu   &=\re^{-\nu_\rsu L} e^{\rsu,\rl}_{1,\infty} + \mc{O}(\re^{(r_\rsu - \rho)L}),\notag\\
  \Phi_2^\rsu(-L,0) v_2^\rsu   &=\re^{-\overline{\nu_\rsu} L} \overline{e^{\rsu,\rl}_{1,\infty}} + \mc{O}(\re^{(r_\rsu - \rho)L}).
\end{align}
 
 Hypothesis \ref{h:inc} then gives that there exists constants $c_1,c_1' \in \C$ not both zero such that 
\begin{align}
P_{2,+}^{\rsu}v_1^{\rss} = \mc{P}^\rl v_1^{\rss}
&=c_1 v_1^\rsu + c_1' v_2^\rsu.\notag
\end{align}
Since $\mc{P}^\rl$ is an isomorphism which commutes with the action $T_1(\theta)$, it can then be obtained that $\mc{P}^\rl$ is complex-linear so that $c_1'= 0$.  Combing this all together we obtain,
\begin{align}
\Phi_2^\rsu(-L,0)P_{2,+}^{\rsu}(0)\Phi_2^\rss(0,-L) e_{1,\infty}^{\rss,\rl} &= 
 c_1 \re^{(\nu_\rss-\nu_\rsu) L}e_{1,\infty}^{\rsu,\rl}   + \mc{O}(\re^{-(\Delta r + \rho)L }).
\end{align}
The expansion for $\overline{e_{1,\infty}^{\rss,\rl}}$ follows in an analogous way for some constant $c_2\in \C$.

\end{proof}

In a similar manner, we can also obtain expansions for bounded solutions of the adjoint variational equation along $q_1$.
\begin{Lemma}\label{l:exadj}
Let $q_1(0)$ be sufficiently close to $U_*$, $L$ sufficiently large, and $E_1^\perp=\mathrm{span}\{e_{j,0}^*\}_{j=1,2}$. Then, for some $\rho>0$, there exists a complex eigenvector $e_{j,\infty}^*$ of $(D_U F(U_*))^*$ with eigenvalue $-\overline{\nu_\rsu}$ and $|e_{j,\infty}^*| = 1$,  such that the bounded solutions of the adjoint equation satisfy
$$
e_{j}^*(L) =\Phi_1^{*,\mathrm{us}}(L,0) e_{j,0}^* = \lp(\tl c_j \re^{-\overline{\nu_\rsu} L}e_{j,\infty}^*+\mathrm{c.c.} \rp)(1+\mc{O}(\re^{-\rho L})),
$$
for some constants $\tl c_j\in \C$, where $\Phi_1^{*,\mathrm{us}}$ denotes the dichotomy of the adjoint variational equation \eqref{e:adjvar} associated with the spectral set $\{\nu\,:\, \rre\{ \nu\} \leq r_\rsu\}$.
\end{Lemma}

Before completing the analysis of the bifurcation equation, we need one more prepatory lemma which estimates the scalar product contained inside of \eqref{e:mel4}.
\begin{Lemma}\label{l:scp}
There exists a $\rho>0$ such that for all $L$ sufficiently large, and $j = 1,2$ we have the following expansion
\beq\label{e:scp}
\la \mf{u}_1, e_j^*(L)\ra = \la w_2^\rsu(-L), e_j^*(L)\ra  + \mc{O}(\re^{-(2\Delta \eta + \rho)L}).
\eeq
\end{Lemma}
\begin{proof}
Applying $P_2^\rss(-L)$ to \eqref{e:glue}, we use the asymptotic expansion of Hypothesis \ref{h:fa}, and the projection estimates in Lemma \ref{l:proj} to obtain
\begin{align}
\mf{s}_2 &= -P_2^\rss(-L)\Delta q(L) + P_2^\rss(-L) w_1(L) \notag\\
&= a \re^{\nu_\rss L} e_{1,\infty}^{\rss,\rl} + c.c. + \mc{O}(\re^{-(r_\rss + \Delta \eta)L}).  \label{e:s2}
\end{align}
This, combined with the result of Lemma \ref{l:exadj}, and the fact that $\la e_{1,\infty}^{\rss,\rl}, e_{j,\infty}^*\ra = 0$ ,  allows us to obtain the estimate
$$
|\la w_2^\rss(-L), e_j^*(L) \ra | \leq K\re^{- (2\Delta \eta + \delta_\rss)L},
$$
for some constant $K>0$.
Also, once again using the projection estimates in Lemma \ref{l:proj}, we obtain
\begin{align}
\la \Delta q(L), e_j^*(L) \ra &= \la q_2(L), e_j^*(L) \ra (1 + \mc{O}(\re^{-\Delta \eta L})),\notag\\
&\leq K\re^{(r_\rsu - r_\ru)L}, \notag\\
&\leq K\re^{-(2\Delta \eta + \rho)L},\notag
\end{align}
where this last inequality comes from the eigenvalue requirements in Hypothesis \ref{h:eig1}(i).

Finally, we use the gluing equation \eqref{e:glue}, and the fact that $e_j^*(\xi)\perp E_1^\rss(\xi)$, to find 
\begin{align}
\la \mf{u}_1, e_j^*(L)\ra &= \la w_1^\rsu(L), e_j^*(L)\ra \notag\\
&= \la w_2^\rsu(-L) + w_2^\rss(-L) - w_1^\rss(L) - \Delta q(L), e_j^*(L)\ra\notag\\
&= \la w_2^\rsu(-L) + w_2^\rss(-L) - \Delta q(L), e_j^*(L)\ra,\notag
\end{align}
which, combined with the above estimates yields the result.
\end{proof}

We may now complete the proof of the main theorem with the following proposition which gives the leading order expansion of the right-hand side of \eqref{e:mel4}.
\begin{Proposition}\label{p:bif}
The bifurcation equation \eqref{e:mel5} has the following leading order expansion in $\mu$,
\begin{align}
\mc{M} \mu =
 -\sum_{j=1,2} \lp(\re^{2\Delta \nu L} d_j + \mathrm{c.c.}\rp) e_{j,0}^*  + \mc{O}(\re^{-(2\Delta r+\rho) L})
+\mc{O}\lp(|\mu|\lp( |\mf{v}_1| + |\mu| \rp) \rp),
\end{align}
where 
$$
d_j = a c_1 \overline{\tl c_j}  \la  e_{1,\infty}^{\rsu,\rl},  e_{j,\infty}^*\ra_\C,
$$
for $a,c_1, \tl c_j\in \C$ as defined in Hypothesis \ref{h:fa}, Lemma \ref{l:ex}, and Lemma \ref{l:exadj} respectively, and where $\la\cdot,\cdot\ra_\C$ is the complexified inner product induced by the real inner product on $X$.
\end{Proposition}

\begin{proof}
To begin, by applying the projection $P_{2,+}^\rsu(0)$, and its complement $I - P_{2,+}^\rsu(0)$ to the transverse matching equation \eqref{e:tr}, we obtain
$$
\mf{u}_2 = w_2^\rsu(0) = -P_{2,+}^\rsu(0)w_2^\rss(0) + \mc{O}(\re^{-(\eta_\rss +\Delta \eta)L}).
$$

Then, using estimates similar to those in the proof of Proposition \ref{p:fp} we find
\begin{align}
w_2^\rsu(-L)& = \Phi_2^\rsu(-L,0) \mf{u}_2 + \mc{O}(\re^{-(2\Delta\eta +\rho)L})\notag\\
w_2^\rss(0)& = \Phi_2^\rss(0,-L) \mf{s}_2 + \mc{O}(\re^{-(r_\rss L + 2\Delta \eta L)}).\notag
\end{align}
We then combine these estimates to obtain,  
\begin{align}
w_2^\rsu(-L) =& \Phi_2^\rsu(-L,0) \mf{u}_2 + \mc{O}(\re^{-(2\Delta\eta +\rho)L})\notag\\
=& -\Phi_2^\rsu(-L,0)P_{2,+}^\rsu(0)w_2^\rss(0)+\mc{O}(\re^{-(2\Delta\eta +\rho)L})\notag\\
=&  -\Phi_2^\rsu(-L,0)P_{2,+}^\rsu(0) \lp[ \Phi_2^\rss(0,-L) \mf{s}_2 + \mc{O}(\re^{-(r_\rss L + 2\Delta \eta L)}) \rp] +\mc{O}(\re^{-(2\Delta\eta +\rho)L})\notag\\
=& -\Phi_2^\rsu(-L,0)P_{2,+}^\rsu(0)\Phi_2^\rss(0,-L) \mf{s}_2+\mc{O}(\re^{-(2\Delta\eta +\rho)L})\notag\\
=& -\Phi_2^\rsu(-L,0)P_{2,+}^\rsu(0)\Phi_2^\rss(0,-L) \lp( a \re^{\nu_\rss L} e_{1,\infty}^{\rss,\rl} + c.c.  \rp) + \mc{O}(\re^{-(2\Delta\eta + \rho)L})\notag\\
=& \,a c_1 \re^{(2\nu_\rss - \nu_\rsu)L}   e_{1,\infty}^{\rsu,\rl} + 
a c_2\re^{(2\overline{\nu_\rss} - \overline{\nu_\rsu})L}\    \overline{e_{1,\infty}^{\rsu,\rl}}  + \mc{O}(\re^{-(2\Delta \eta + \rho)L}),\notag
\end{align}
where estimate \eqref{e:s2} was used in the fifth line, and Lemma \ref{l:ex} used in the sixth.  Since this last expression must be real (being the flow of a real initial condition), it can be found that $c_2 = \overline{c_1}$. 
Hence we obtain 
\beq
w_2^\rsu(-L)  = a c_1 \re^{(2\nu_\rss - \nu_\rsu)L} e_{1,\infty}^{\rsu, \rl}  + c.c + \mc{O}(\re^{-(2\Delta\eta+\rho)L}).
\eeq
Applying Lemma \ref{l:exadj} to \eqref{e:mel4}, we obtain
\beq\label{e:bifpf2}
\la \Phi_1^\rsu(0,L)\mf{u}_1,e^*_{j,0}\ra_\C = \la \mf{u}_1, e^*_{j}(L)\ra_\C =\la \mf{u}_1, \re^{-\nu_\rsu L}\tl c_j\,e^*_{j,\infty} +\mathrm{c.c.} \ra_\C(1+\mc{O}(\re^{-\rho L})).
\eeq
 Finally by substituting the expansion obtained for $\mf{u}_1$ in Lemma  \ref{l:scp},  and taking into account the fact that $$\la e_{1,\infty}^{\rsu,\rl}, \overline{e^*_{1,\infty}}\ra_\C = 0,$$ we obtain the result.

\end{proof}

\section{Discussion}\label{s:dis}


\subsection{Application of results}\label{ss:ar}
We now discuss the applicability of our result in the examples given in Section \ref{s:ex}.  

\begin{paragraph}{Cubic-quintic complex Ginzburg-Landau equation}

In this example, all of the required hypotheses either have been proven in previous works, or can be proven using straightforward techniques.  As we study real equations above, one must first write \eqref{e:cgl3} in terms of the variables $\rre u,\rim u$ or $u,\overline{u}$, obtaining a four-dimensional real system.  We also note that the well-posedness assumption of Hypothesis \ref{h:ex} is trivially satisfied.  

Next it can readily be calculated for parameters as in the following proposition, that the eigenvalues of the asymptotic linearization of \eqref{e:cgl3} with $\xi = +\infty$ about the equilibrium $u_*\equiv0$ is hyperbolic, with a complex conjugate pair of eigenvalues on each side of the imaginary axis.  Furthermore, for $\xi = -\infty$, the spectrum of the linearization consists of complex conjugate pairs $\nu_\rss,\overline{\nu_\rss}$, and $\nu_\rsu,\overline{\nu_\rsu}$ which satisfy the desired hypotheses. The results of \cite{vanSaarloos92} then give the following proposition 
\begin{Proposition}\label{p:vS}
For $\alpha,\beta,\gamma$ sufficiently small, and $\rho>1$, there exists a pushed front solution $u_\rff$ of the form $u_\rff(\xi,t) = \re^{\ri \omega_\mathrm{p} t}u_\mathrm{f}(\xi)$ which invades $u_*$ with speed $c_\mathrm{p}>c_\rlin$ and some angular frequency $\omega_\mathrm{p}$.  Here, $u_\mathrm{f}$ solves \eqref{e:cgl2}, approaches the periodic pattern $u_\mathrm{p} = r_\mathrm{p}\re^{\ri k_\mathrm{p} \xi}$ as $\xi\rightarrow-\infty$, and approaches $u_*$ as $\xi\rightarrow+\infty$, where $c_\mathrm{p},\omega_\mathrm{p}, r_\mathrm{p}$, and $k_\mathrm{p}$ satisfy the nonlinear dispersion relation \eqref{e:disp}.  Furthermore, the periodic orbit, $u_p$, has two-dimensional center-unstable manifold in the flow defined by \eqref{e:cgl2}.
\end{Proposition}

\begin{Remark}\label{r:coor}
Note that our parameter assumptions differ slightly from those of \cite{vanSaarloos92} where $\rho$ is scaled to be equal to 1, and the coefficient of the linear term $u$ is small.  In order to go from our parameters to theirs, one should make the scalings
$$
 u = \fr{\tl u}{a},\quad x = \fr{\tl x}{a^2},\quad  t = \fr{\tl t}{ a^4},\quad c = a^2 \tl c, \quad \chi_\epsilon =a^4 \tl \chi_\epsilon, \quad\gamma = a^2\tl \gamma,\quad a^2 = \rho.
$$
\end{Remark}

Next, for the trigger $\chi_0$ given in \eqref{e:cgl3} above, or for a step-function trigger satisfying $\chi_0\equiv \pm 1$ for $\xi \lessgtr 0$, we have that the trivial solution $u_*$ is a preparation front, and that, for the variables $U = (\rre\, u,\rre \,u_\xi, \rim \,u, \rim \,u_\xi,\chi_0)$, the spatial dynamics equilibria $U_* =(0,0,0,0,-1),\tl U_* = (0,0,0,0,1)$ satisfy $W^\rsu(U_*) = 3$ and $W^\rs(\tl U_*) = 3$, where one dimension from each count is from the $\chi$ direction. Also, let $U_\mathrm{ff}$ denote the heteroclinic orbit in this formulation which corresponds to $u_\mathrm{ff}$.  For $\chi_0$, the tangent spaces of these invariant manifolds are constant and can be explicitly calculated in terms of the spatial eigenvalues. The desired transversality of the intersection about $u_0$ can then be obtained by calculating that
$$
\det \left(\begin{array}{cc}1 & 1 \\\nu_2^+ & \nu_1^-\end{array}\right)\neq 0,
$$ 
where $\nu_j^\pm$ solve the dispersion relation, $d_\pm(\nu) = (1+\ri \alpha)\nu^2+c \nu + (\pm 1 - \ri \omega)$, and are ordered by increasing real part. A standard singular-perturbation argument then gives the transversality for $\chi_\epsilon$ with $0<\epsilon\ll1$.

All that is left is to verify are the intersection properties along $U_\rff$, and the invertibility of the associated Melnikov matrix.  Note that since there is no non-leading strong-stable eigenspace, $E_{1,\infty}^{\rss,\rs}$, the inclination assumption in Hypothesis \ref{h:inc} is trivially satisfied.
\begin{Proposition}\label{p:cglM}
For $\alpha, \beta, \gamma,\epsilon$ sufficiently small, the tangent spaces satisfy 
$$
\dim T_{U_\rff(\xi)}W^\rss(U_*) \cap T_{U_\rff(\xi)}W^\rcu(U_\mathrm{p}) =  \dim \lp( T_{U_\rff(\xi)}W^\rss(U_*) + T_{U_\rff(\xi)}W^\rcu(U_\mathrm{p}) \rp)^\perp = 2,
$$
and the Melnikov matrix 
$$
M =\left(\begin{array}{cc} \int_{-\infty}^\infty\la \psi_1(\xi),\p_c F(U_\rff(\xi)\ra d\xi & 
\int_{-\infty}^\infty\la \psi_1(\xi),\p_\omega F(U_\rff(\xi)\ra d\xi   \\ 
\int_{-\infty}^\infty\la \psi_2(\xi),\p_c F(U_\rff(\xi)\ra d\xi  & 
\int_{-\infty}^\infty\la \psi_2(\xi),\p_\omega F(U_\rff(\xi)\ra d\xi \end{array}\right), 
$$
 is invertible, where $\psi_j(\xi)$ are solutions of the adjoint variational equation of the linearization of \eqref{e:cgl3} with initial conditions $\psi_j(0)$ satifying 
 $$
 \mathrm{span}_{j=1,2}\{\psi_j(0)\} = \lp( T_{U_\rff(\xi)}W^\rss(U_*) + T_{U_\rff(\xi)}W^\rcu(U_\mathrm{p}) \rp)^\perp.
 $$
\end{Proposition}
\begin{proof}
First, after scaling $t = \omega \tau$, note that $\p_\tau u_\rff$ and $\p_\xi u_\rff$ are linearly independent and lie in the kernel of the linear operator
\begin{align}
\mc{T}:V\mapsto& - \omega V_\tau + (1+\ri \alpha) V_{\xi\xi} + c V_{\xi} +V + g_{A,\overline{A}}(A_\rff,\overline{A_\rff}) V.
\end{align}

Since $\epsilon$ is small, it can be readily obtained that $\chi_\epsilon$ does not induce any eigenfunctions corresponding to resonance poles of the Evans function about $u_*$ and hence that $\p_\tau u_\rff,\,$ and $\p_\xi u_\rff$ span the kernel, so that $0$ is an eigenvalue of $\mc{T}$ with geometric multiplicity $2$. We claim its algebraic multiplicity is also equal to 2.  Momentarily assuming this claim, the first statement of the proposition follows immediately, as the adjoint variational equation is found to have two linearly independent bounded solutions.  The proofs of Theorem 8.4 and Lemma 8.7 in \cite{Sandstede01} can then be used to obtain that $M$ is invertible.

To prove the claim, we first note that in the real-coefficient case $\alpha =\gamma  = \beta = 0$ the linearized operator $\mc{T}$ is self-adjoint when defined on an exponentially weighted space $L^2_{c/2}(\R\times \mb{T})$ with weight $\re^{\fr{c}{2}\xi}$.  This implies that $\mc{T}$ does not have a generalized kernel for $\alpha = \gamma = \beta = 0$.  Since algebraic simplicity is an open property, we then have that for complex parameters $\alpha, \gamma, \beta$ sufficiently small, the eigenfunctions $\p_t u_\rff, \p_\xi u_\rff$ remain algebraically simple.
\end{proof}

With these propositions in hand we can then apply our results to obtain the existence of a family of pushed trigger fronts in the modified qcGL equation which bifurcate from the pushed free front solution obtained in Proposition \ref{p:vS}.

\begin{Remark}
We also note that our results could be obtained for qcGL with the step function trigger $\chi_0$ by first simplifying the phase-space with the blow-up coordinates used in \cite{GohScheel14}.  One then obtains a phase space $\R_+\times \mc{S}^2$, where $\mc{S}^2$ denotes the Riemann sphere.  In these coordinates, the dynamics are smoothly foliated over $\mc{S}^2$ which is a normally hyperbolic invariant manifold.  Furthermore, the dynamics on $\mc{S}^2$ are described by a Ricatti equation which can be explicitly integrated.  Here, the periodic orbit $u_p$ reduces to a point, with one-dimensional unstable manifold, while the target manifold $W_+^\rs(0)$ is two dimensional. Lin's method could then be readily applied to obtain the  desired result.\end{Remark}

\end{paragraph}

\begin{paragraph}{Cahn-Hilliard}

In the case of Cahn-Hilliard, more work needs to be done to apply our bifurcation result.  While the existence of a preparation front $u_\rpr$ can be obtained using a Conley index argument (see \cite[App. A]{goh15}), the existence of an oscillatory pushed front is, to the authors' knowledge, still an open problem.  Furthermore, the spectrum of the linearized modulated traveling wave problem would have to be obtained by first using a Fourier decomposition in time,
$$
\ri \omega \ell \hat{u} = -\p_{\xi\xi}\lp( \p_{\xi\xi} \hat{u} + f'(u_*)\hat{u}  \rp) = c\p_\xi \hat{u} ,\quad \ell\in \Z,
$$
writing each equation as a first order equation in $\xi$, and finding the four spatial eigenvalues $\nu_\ell$ which satisfy
$$
0=\nu^4 + f'(u_*)\nu^2 - c_\mathrm{p} \nu + \ri\omega_\mathrm{p}\ell, \quad \ell\in \Z.
$$  
  Using a scaling argument, it is possible to show that such spatial eigenvalues are bounded far away from the imaginary axis for large $\ell$ and thus only a few values of $\ell$ need be studied.   From this one should hopefully be able to establish (or assume) the intersection properties of Section \ref{ss:int}, possibly after factoring out the $S^1$-equivariance, and obtain a leading order expansion for the bifurcation equation. 

In practice, one may also proceed by verifying the hypotheses with numerical computations.  As shown above, the existence of a pushed front can be evidenced by numerical continuation.  Then, the spaitial eigenvalues $\nu_\ell$ could be found for each $\ell$ using the values for $c_\mathrm{p}$ and $\omega_\mathrm{p} $ obtained from the AUTO calculation in Figure \ref{f:chcont}.   Regarding the discussion in Remark \ref{r:fsym}, we note that for the numerically determined $c_\mathrm{p}$ and $\omega_\mathrm{p}$, the leading eigenspaces corresponding to $\nu_{\rss/\rsu}$ lie in the $\ell=1$ Fourier subspace.

One could then use a numerical eigenvalue solver to test the transversality hypotheses on $u_\rpr$ and $u_\mathrm{ff}$.  For the former, one need only verify that the kernel of the discretized linearization about $u_\rpr$ is empty.  For the latter, since the free pushed front is time-periodic in a co-moving frame, one must look at the spectrum of the discretized linear period map and determine that the algebraic multiplicity of the Floquet exponent at 0 is two.  Finally, since inclination-flip configurations are degenerate, the failure of the inclination hypothesis could be tested by perturbing the preparation or free pushed fronts.  

In many experiments and models using the Cahn-Hilliard equation (see for example \cite{Foard12,goh15,lagzi13}), the preparation front is controlled by a traveling source term, instead of a spatial inhomogeniety as in \eqref{e:mCH} above.  Such systems usually take the form
\beq\label{e:chsrc}
u_t = -(u_{xx} + f(u))_{xx} + cu_x + c h(x),
\eeq
where the source term $h:\R\rightarrow \R$ is positive, spatially localized, and deposits mass into the system to transform a stable homogeneous equilibrium into to an unstable state in its wake. For simplicity, let us also assume that $h'$ is compactly supported.  To apply our results in this case we must take into account that this equation preserves mass and hence has a linearization with additional neutral modes. This is manifested in the corresponding spatial dynamics formulation,
\begin{align}
u_x &=v \notag\\
v_x &= \theta - f(u) \notag\\
\theta_x &= w\notag\\
w_x& = -\omega \p_\tau +c u_x + h(\xi),
\end{align}
as the existence of a conserved quantity
$$
I(u_1;c) =\fr{1}{2\pi}\int_0^{2\pi} w - cu \,\,d\tau, \qquad u_1 = (u,v,\theta,w),
$$
which is constant under the flow for all $|\xi|$ sufficiently large (i.e. outside of the support of $h'$).  The existence of such a quantity implies the existence of a family of periodic orbits and hence pushed free invasion fronts which are parameterized by fixed values $I(u_1;s)\equiv m$.  Thus, pushed free invasion fronts, if shown to exist, will come in a 1-parameter family as well. One can obtain existence of pushed trigger fronts, by restricting the phase-space to the affine, co-dimension one subspaces $\{I \equiv m\}$ and then applying our results.

More generally, if the spatial dynamics formulation of a pattern-forming system possesses conserved quantities one must perform a dimension counting to check that our genericity and intersection hypotheses still hold.  Namely, one must verify that the introduction of neutral modes about the equilibria and periodic orbit preserves the transversality and Fredholm properties we require. For more information and an example of such calculations see \cite[Sec. 4]{Goh11}.

\end{paragraph}

\subsection{Other spectral splittings}
As mentioned in the introduction, the spectral splitting associated with the pushed front's strong-stable decay and the next weakly-stable eigenvalue comes in other varieties in addition to the case we studied.  First we remark that a system where $\nu_\rss$ has a complex conjugate while $\nu_\rsu$ is real and simple should behave in the same way as discussed above, as the quantity $\Delta\nu$ would still be complex.  One such example arrises in the Extended Fisher-Kolmogorov equation,
\beq\label{e:efk}
u_t = -\gamma\p^4_x u + \p_x^2 u + \fr{u}{b}(b+u)(1-u),\quad x,t\in \R,\,\, u\in \R.
\eeq 
It has been observed in \cite{van1989} that pushed fronts exist for $b$ sufficiently small.  For $\gamma<1/12$, this front is asymptotically constant in the wake, while for $\gamma>1/12$ it forms a spatially periodic pattern of ``kinks" and ``anti-kinks".  In the former case, with $\gamma = 0.08$ for example, the pushed speed is found to be $c_\mathrm{p} \approx 2.175$ while  spatial eigenvalues $\nu_\rss,\overline{\nu_\rss} \approx -0.575971 \pm 1.21251\,\ri$ and $\nu_\rsu \approx -0.365678$.  Hence, the pushed front in this case has an oscillatory tail and we thus predict that a triggered version of this equation, with say $\chi(x - ct)$ multiplied by the linear term $u$, would exhibit non-monotone front locking with respect to the trigger interface. 

Many examples arise where both $\nu_\rss$ and $\nu_\rsu$ are real, generically leading to monotone front selection and no front locking phenomena unless a more complicated spatial trigger is introduced.   One such example is the cubic-quintic Nagumo equation,
\beq\label{e:nag}
u_t = \p_x^2 u + u + d u^3 - u^5, \quad x,t\in \R, \,\, u\in \R.
\eeq 
Here, using a reduction of order method (see \cite{van1989}), one finds that free pushed fronts exist for all $d>\fr{2\sqrt{3}}{3}$ and travel with speed $c_\mathrm{p}=\fr{-d + 2\sqrt{d^2+4}}{\sqrt{3}}$.  

While such fronts will always be asymptotically constant (i.e. no periodic pattern in the wake), we hypothesize that certain spatial triggers could induce the front locking phenomena discussed above. To this end, one could explore triggering phenomena by moving into a co-moving frame of speed $c$ and studying the equation on a semi-infinite domain $x\in (-\infty,0]$ with various boundary conditions $B_1(u_x) + B_2(u) = 0$ at $x = 0$.  In order for the problem to be well-posed, one would look at conditions of the form $u_x(0,t) = \phi(u(0,t))$, for some smooth function $\phi:\R\rightarrow\R$.  Triggered fronts would then be obtained by finding connections between the strong-stable manifold, $W^\rss(0)$, and the boundary manifold $\mc{B}^\phi$ defined by the graph of $\phi$.   By selecting specific boundary conditions, one could then obtain multi-stability of fronts which lock to the boundary condition at different distances. See \cite{morrissey15} for a general study of this subject in the case where the co-moving frame speed is zero.  We also remark that it may be possible to observe interesting dynamics if a triggering mechanism could be used to perturb pushed fronts in the phase-field system studied in \cite{Goh11}.

\subsection{Stability of pushed trigger fronts}\label{ss:inst}
Though we did not study the stability of pushed trigger fronts, we expect such solutions to be stable for parameters lying on branches of the bifurcation curve $\mu_*(L)$.  
For more general types of triggers, resonance poles or branch poles ahead of the preparation front could induce faster speeds and different wavenumbers in the wake. The solutions we construct here would in such a case be unstable. These effects have been documented in a particular, prototypical case of coupled KPP equations in \cite{holzer14b}.  In the context of the examples given in Section \ref{s:ex},  if $\epsilon$ were not small, the interface of $\chi_\epsilon$ would be shallow, taking a long time to ramp up from -1 to 1.  This would cause resonance poles to arise in the linearization about $u_*$ leading to instabilities in the interfacial region where $\chi_\epsilon$ is not close to $\pm1$.  Such an example could also be realized in \eqref{e:chsrc} by making the source term $h$ only weakly localized. Here, the resulting preparation front would possess localized instabilities as the interface slowly passes through the spinodal region.

Additionally, different patterns would be selected if the triggering function $\chi$ were not monotone. For example, if there were a bounded region where $\chi>1$ then the linearization about the unstable equilibrium could possess unstable extended point spectrum which would effect the pattern-selection mechanism.  Also, see \cite{wilzek14} for interesting numerical results where spatially periodic forcing induces the selection of different patterns and locking behavior.

%


\bibliography{PushedLocked}
\bibliographystyle{siam}

%
%
%
%
%
%
%
%
%
%
%
%
%
%
%
%
%
%
%

%
%
%
%
%
%
%
%
%
%
%
%
%

\end{document}